\documentclass[11pt]{article}

\usepackage{amsmath, amsthm, amsfonts, amsxtra, fullpage, color, hyperref, url}
\usepackage{epsfig}
\usepackage{graphicx, psfrag}
\usepackage[inline]{enumitem}
\usepackage{listings}

\allowdisplaybreaks

\definecolor{refkey}{gray}{.5}   

\definecolor{labelkey}{gray}{.5} 

\numberwithin{equation}{section}

\newcommand{\di}{\displaystyle}

\newcommand{\R}{{\mathbb R}}
\newcommand{\C}{{\mathbb C}}
\newcommand{\N}{{\mathbb N}}

\newcommand{\Z}{{\mathbb Z}}
\newcommand{\T}{{\mathbb T}}
\renewcommand{\d}{\partial}
\renewcommand{\Re}{{\operatorname{Re\,}}}
\renewcommand{\Im}{{\operatorname{Im\,}}}

\newcommand{\supp}{{\operatorname{supp}\,}}
\newcommand{\Ai}{{\operatorname{Ai}}}
\newcommand{\Bi}{{\operatorname{Bi}}}

\newcommand{\arctanh}{{\operatorname{arctanh}}}
\newcommand{\I}{{\mathbf I}}

\newcommand{\Res}{{\operatorname{Res}\,}}
\newcommand{\al}{\alpha}
\newcommand{\be}{\beta}
\newcommand{\ga}{\gamma}
\newcommand{\Ga}{\Gamma}

\newcommand{\ep}{\varepsilon}
\newcommand{\de}{\delta}
\newcommand{\De}{\Delta}

\newcommand{\sg}{\sigma}
\newcommand{\Sg}{\Sigma}
\newcommand{\om}{\omega}
\newcommand{\Om}{\Omega}
\renewcommand{\th}{\theta}
\newcommand{\z}{\zeta}

\newcommand{\bigO}{{\mathcal O}}
\newcommand{\mcal}{{\mathcal M}}

\newcommand{\lv}{{\langle}}
\newcommand{\rv}{{\rangle}}

\newtheorem{thm}{Theorem}[section]
\newtheorem{cor}[thm]{Corollary}
\newtheorem{lem}[thm]{Lemma}
\newtheorem{prop}[thm]{Proposition}
\theoremstyle{remark}
\newtheorem{rem}{Remark}[section]

\newcommand{\eq}{\textrm{eq}}
\newcommand{\lattice}{L_{\al, m}}
\newcommand{\beq}{\begin{equation}}
\newcommand{\eeq}{\end{equation}}

\title{The Fourier extension method and discrete orthogonal polynomials on an arc of the circle }

\author{J.S. Geronimo\thanks{Department of Mathematics, Georgia Institute of Technology, Atlanta, USA. email: \href{mailto:geronimo@math.gatech.edu}{\nolinkurl{geronimo@math.gatech.edu}}} \and 
Karl Liechty\thanks{Department of Mathematical Sciences, DePaul University, Chicago, IL, 60614 USA \href{mailto:kliechty@depaul.edu}{\nolinkurl{kliechty@depaul.edu}}.}
}

\begin{document}

\maketitle

\begin{abstract}
The Fourier extension method, also known as the Fourier continuation method, is a method for approximating non-periodic functions on an interval using truncated Fourier series with period larger than the interval on which the function is defined. When the function being approximated is known at only finitely many points, the approximation is constructed as a projection based on this discrete set of points. In this paper we address the issue of estimating the absolute error in the approximation. The error can be expressed in terms of a system of discrete orthogonal polynomials on an arc of the unit circle, and these polynomials are then evaluated asymptotically using Riemann--Hilbert methods. 
\end{abstract}

\bigskip


\medskip


\section{Introduction}\label{intro}

Let $f:[-1/2,1/2] \to \C$ be a smooth function which we would like to approximate via truncated Fourier series of length $M\in \N$. The usual Fourier methods would involve projecting $f$ onto the  space of Laurent polynomials of period $1$ given by 
\begin{equation}
S_M={\rm span}\{ e^{2 \pi i  k x}\}_{k\in t(M)},
\end{equation}
 where 
 \begin{equation}\label{def:tM}
 t(M)=\left\{
 \begin{aligned}
& \left\{k\in\Z : -\frac{(M-1)}{2}\le k\le\frac{(M-1)}{2} \right\}  \ {\rm for} \  M \ {\rm odd} \\
& \left\{k\in\Z : -\frac{ M   }{2}\le k\le\frac{M}{2}-1\right\}  \ {\rm for} \ M \ {\rm even}.
\end{aligned}\right.
\end{equation}
If $f$ is not periodic, then the usual Fourier methods fail to give a uniform approximation of $f$ near the endpoints $\pm 1/2$ due to the Gibbs phenomenon. One method for dealing with this problem is to extend $f$ smoothly to a function $\tilde{f}$ on a larger interval $[-b/2,b/2]$ ($b>1$), and to try to approximate $\tilde{f}$ (and therefore $f$ itself) by a Fourier series on this larger interval. It is well known that such a smooth periodic extension is always possible, although it is far from unique. This method is known as the {\it Fourier extension method} \cite{Boyd02} or the {\it Fourier continuation method} \cite{Bruno03, brhapo}, and involves projection onto the space of Laurent polynomials of period $b>1$ given by 
\begin{equation}
{S}^b_M=\{ e^{\frac{2 \pi i  k}{b} x}\}_{k\in t(M)},
\end{equation}
where $t(M)$ is as in \eqref{def:tM}. Since the Fourier extension $\tilde{f}$ is a smooth periodic function with period $b$, it has a Fourier series
\begin{equation}\label{ftilde_Fourier}
\tilde{f}(x) = \sum_{k=-\infty}^\infty a_k e^{2\pi i k x/b},
\end{equation}
where the coefficients $a_k$ decay faster than any power of $k$. 

We consider the Fourier extension problem of the third kind as described by Boyd \cite{Boyd02}, in which the function $f$ is not known outside of its original interval of definition $[-1/2,1/2]$. We assume that $f$ is known only at a finite number of points, which we assume for now are equispaced on the interval $[-1/2,1/2]$. That is, for a given positive integer $N$, the value of $f(x)$ is known only at the points $x_1, \dots, x_N$ given by 
 \begin{equation}\label{def:xj}
x_j=\frac{j}{N}-\frac{1}{2}-\frac{1}{2N}, \quad j=1,2,\ldots, N.
\end{equation}

In order to describe the projection onto ${S}^b_M$ based on this data, define the discrete inner product $\lv \cdot, \cdot \rv_N$ as
 \begin{equation}\label{def:inner_product}
 \lv f, g\rv_N := \frac{1}{N} \sum_{j=1}^N f(x_j) \overline{ g(x_j)},
\end{equation}
and let $|| \cdot ||_N$ be the norm inherited from this inner product. We seek the function $q\in {S}^b_M$ which is closest to $f$ in this norm. That is, $q$ should satisfy
\begin{equation}\label{ls}
||f-q||_N^2=\min_{h\in S^b_M}||f-h||_N^2,
\end{equation}
and it is a simple exercise in linear algebra to see that the minimizer is unique provided $M\le N$, which we assume throughout this paper.

Here we note that the minimization problem \eqref{ls} comes from a discrete inner product, but one would like to have uniform bounds on the difference $|f(x)-q(x)|$ for all $x\in[-1/2,1/2]$. Obtaining $L^\infty$ bounds from a discrete $L^2$ construction is a delicate issue which we address in this paper. When the discrete inner product \eqref{def:inner_product} is replaced by the usual $L^2$ inner product, the analysis is simplified greatly and an exponential bound on the $L^\infty$ norm of the difference $f-q$ was obtained in \cite[Theorem 2.3]{Adcock-Huybrechs14}, see also \cite{Huybrechs10, Webb-Coppe-Huybrech18}.  Recent works have distinguished between the {\it discrete Fourier extension} defined in terms of the discrete inner product \eqref{def:inner_product} and the {\it continuous Fourier extension} defined in terms of the usual $L^2$ inner product \cite{Adcock-Huybrechs14, adhum}. While the continuous Fourier extension is of theoretical interest, the discrete Fourier extension arises naturally in applications, and plays an important role in the solution of numerical PDEs, see e.g. \cite{adhum, brhapo} and references therein. In the current paper we study the absolute error for the discrete Fourier extension (in one dimension) based on the approach outlined in \cite[Section 2.3]{brhapo}. Namely, the error may be expressed as a series involving the Fourier coefficients of the extension function $\tilde f$ and a certain sequence of functions $\{B^k_{N,M}(x)\}_{k\in \Z}$ which is independent of $f$, see equation~\eqref{def:Bk} below. The primary results of this paper give uniform asymptotic estimates on the functions  $B^k_{N,M}(x)$ for large $M$ and $N$, with $N/M>1$ bounded, see Theorems \ref{main_band},  \ref{upper}, and \ref{lower}, as well as Corollary \ref{maincor}. The proofs are based on the fact that $B^k_{N,M}$ can be expressed in terms a system of discrete orthogonal polynomials on an arc of the unit circle, and these orthogonal polynomials may be evaluated asymptotically using Riemann--Hilbert methods. A key ingredient in the analysis is a constrained equilibrium problem (see Section \ref{heuristics}), which was first introduced by Rakhmanov  in the study of Chebyshev polynomials of a discrete variable \cite{Rakhmanov96}. For a more general study of constrained equilibrium problems, see the paper \cite{Dragnev-Saff97} of Dragnev and Saff. In the context of numerical analysis, constrained equilibrium problems have appeared in analysis of convergence properties of Krylov subspace iterations, see the survey paper \cite{Kuijlaars06} of Kuijlaars and references therein.

Below we express the solution to the minimization problem \eqref{ls} in terms of orthogonal polynomials before presenting our results. 

\subsection{The least squares projection and main results}
Introduce the orthonormal polynomials with respect to the inner product \eqref{def:inner_product} in the variable $z=e^{2\pi i x/b}$. That is, we let $\varphi^{N}_k(z)$ be the polynomial of degree $k$ with positive leading coefficient satisfying
 \begin{equation}\label{def:orthoprod}
 \frac{1}{N}\sum_{j=1}^N \varphi^{N}_k(z_j)\overline{\varphi^{N}_\ell(z_j)} = \de_{k\ell}, \quad z_j = e^{2\pi i x_j/b},
 \end{equation}
 where the sampling points $x_j$ are defined in \eqref{def:xj}. 
The orthogonal projection $P_{ S^b_M}$ onto
$ S^b_M$ above is given by
\begin{equation}\label{OP}
P_{ S^b_M}(f)(x)=\lv f(\cdot),K_{M}(\cdot, x)\rv_N,
\end{equation} 
where 
\begin{equation}\label{km}
K_{M}(x,y)=e^{2\pi i \frac{M_0}{b}(x-y)}\sum_{l=0}^{M-1}\varphi^{N}_l(e^{\frac{2\pi i}{b}x})\overline{\varphi^{N}_l(e^{\frac{2\pi i}{b}y})}, \quad M_0=\min t(M).
\end{equation}
These polynomials only exist for $k<N$, thus $P_{ S^b_M}$ exists for $M\le N$.

To have a good approximation we would like for the error in the projection \eqref{OP} to be uniformly small for $x\in[-1/2, 1/2]$, and we denote the error function as
\begin{equation}\label{minh}
E^{f, b}_{M,N}(x)=(1-P_{S^b_M})(f)(x), 
\end{equation}
which is the same as 
\begin{equation}
(1-P_{S^b_M})(\tilde{f})(x)
\end{equation}
for $x\in[-1/2, 1/2]$, since $f$ and $\tilde{f}$ agree on this interval.
We assume that $\tilde{f}$ has the Fourier expansion \eqref{ftilde_Fourier}, which gives the series for $E^{f, b}_{M,N}(x)$:
\begin{equation}
E^{f, b}_{M,N}(x)=\sum_{k=-\infty}^\infty a_k \left[ e^{2\pi i k x/b} -e^{2\pi i \frac{M_0}{b} x}\sum_{l=0}^{M-1} \varphi^{N}_l(e^{\frac{2\pi
    i}{b}x})\lv e^{-2\pi
  i \frac{M_0}{b} \cdot}e^{\frac{2\pi i k}{b}\cdot},\varphi^{N}_l(e^{\frac{2\pi
      i}{b}\cdot})\rv_N\right],
\end{equation}
or equivalently
\begin{equation}\label{Error_series}
\begin{aligned}
e^{2\pi i \frac{M_0}{b} x}E^{f, b}_{M,N}(x)&=\sum_{k=-\infty}^\infty a_k \left[ e^{2\pi i x(k-M_0)/b} -\frac{1}{N}\sum_{l=0}^{M-1} \varphi^{N}_l(e^{\frac{2\pi
    i}{b}x})\sum_{j=1}^N e^{2\pi
  i \frac{k-M_0}{b} x_j}\overline{\varphi^{N}_l(e^{\frac{2\pi
      i}{b}x_j})} \right] \\
      &=\sum_{k=-\infty}^\infty a_k B^{k}_{N,M}(x),
      \end{aligned}
\end{equation}
where
\begin{equation}\label{def:Bk}
B^{k}_{N,M}(x) := e^{2\pi i x(k-M_0)/b} -\frac{1}{N}\sum_{l=0}^{M-1} \varphi^{N}_l(e^{\frac{2\pi
    i}{b}x})\sum_{j=1}^N e^{2\pi
  i \frac{k-M_0}{b} x_j}\overline{\varphi^{N}_l(e^{\frac{2\pi
      i}{b}x_j})}.
      \end{equation}
The error may therefore be estimated as
\begin{equation}\label{EB_ineq}
\left\lvert E^{f, b}_{M,N}(x)\right\rvert \le \sum_{k\notin t(M)} \left\lvert B^{k}_{N,M}(x)\right\rvert \lvert a_k\rvert.
\end{equation}
Since we assume that $\tilde f(x)$ is a smooth function periodic on $[-b/2,b/2]$, the coefficients $a_k$ decay faster than any power of $k$. The error is therefore uniformly small as long as $\sup_{x\in[-1/2,1/2]} \lvert B^{k}_{N,M}(x) \rvert$ grows no faster than polynomially. We will use the orthonormal polynomials \eqref{def:inner_product} to estimate $\lvert B^{k}_{N,M}(x)\rvert $ for $x\in [-1/2,1/2]$.

Using the Cauchy--Schwarz inequality along with the orthonormality of $\varphi^{N}_l(z)$, we can obtain the simple upper bound 
\begin{equation}\label{ubq}
\begin{aligned}
\left|B^{k}_{N,M}(x)\right| &=\left|e^{\frac{2\pi i(k-M_0)}{b}x}-\sum_{l=0}^{M-1}\varphi^{N}_l(e^{\frac{2\pi
    i}{b}x}) \lv e^{2\pi i\frac{ k-M_0}{b}\cdot}, \varphi^{N}_l(e^{\frac{2\pi
      i}{b}\cdot})\rv_N \right|\\
    &\le 1+\sum_{l=0}^{M-1}\left |\varphi^{N}_l(e^{\frac{2\pi i}{b}x})\right|.
\end{aligned}
\end{equation}
Thus a bound on the orthonormal polynomials gives an upper bound on $\left|B^{k}_{N,M}(x)\right|$. 

 It turns out that the asymptotic estimates for the error terms $B_{N,M}^k(x)$ are vastly different for $x$ close to the middle of the interval $[-1/2,1/2]$ and for $x$ close to the edges. To describe the different behaviors, first define the number $\tilde\be\in(0,\frac{1}{2})$ via the equation
\begin{equation}\label{def:tbeta}
\cos\frac{2 \pi\tilde\be}{b} = \cos(\pi/b) + \left(1+\cos(\pi/b)\right) \tan\left(\frac{\pi M}{2N b}\right)^2.
\end{equation}

\begin{thm}\label{main_band}
Assume $N$ is odd and let $N$ and $M$ approach infinity in such a way that the ratio $N/M\ge 1$ remains bounded. Also fix $b>1$ such that $Nb\in \Z$, let $\tilde\be\in(0,\frac{1}{2})$ be defined by \eqref{def:tbeta}, and 
fix $x$ such that $|x|<\tilde\be$. Then the orthonormal polynomial $\varphi^{N}_M(e^{\frac{2\pi i}{b}x})$ satisfies
\begin{equation}
\varphi^{N}_M(e^{\frac{2\pi i}{b}x}) = \bigO(1),
\end{equation}
as $M\to\infty$. Thus the error term $B_{N,M}^k(x)$ satisfies
\begin{equation}
\left\lvert B_{N,M}^k(x)\right\rvert = \bigO(M).
\end{equation}
These estimates are uniform in $x$ on compact subsets of the interval $(-\tilde\be,\tilde\be)$.
\end{thm}

In the theorem above we make the technical assumptions that $N$ is odd and $Nb$ is an integer. These assumptions are purely technical to ease the analysis, and could be removed with some effort.

When $x$ is outside the interval $(-\tilde\be,\tilde\be )$, the asymptotic formulas for the orthonormal polynomials are quite different. The next two results involve a certain function $\tilde L(x)$ on $[-1/2,1/2]$. This function depends on the extended period $b\ge 1$ as well as the sampling ratio $ N/M\ge 1$. We therefore denote
\begin{equation}
\tilde\xi := N/M,
\end{equation}
and write $\tilde{L}(x) \equiv \tilde{L}(x; b,\tilde\xi)$, suppressing the dependence on $b$ and $\tilde \xi$ when those parameters are fixed and there is no possibility of confusion. An explicit formula for $\tilde L(x; b, \tilde\xi)$ is given in equation \eqref{mr1b}, but it is not essential to the primary results. Some key properties of this function are listed in the following proposition.  
\begin{prop}\label{thm:L_properties}
Let $\tilde{L}(x) \equiv \tilde{L}(x; b, \tilde\xi):= L(2\pi x/b; \pi/b, b\tilde\xi)$ where the function $L(\phi; \al, \xi)$ is defined in \eqref{mr1b}. This function satisfies the following properties.
\begin{enumerate}[label=(\alph*)]
\item \label{thm:L_properties1}For fixed $b\ge 1$ and $\tilde\xi\ge 1$, the function $ \tilde{L}(x)$ is constant for $x\in [-\tilde\beta, \tilde\beta]$, strictly increasing for $x\in [\tilde\beta, 1/2]$, and decreasing for $x\in[-1/2,-\tilde\beta]$.
\item  \label{thm:L_properties2}  As a function of $\tilde{\xi}$, $\tilde L(\pm 1/2; b,\tilde\xi)$ is a decreasing function of $\tilde\xi >1$. The same is true for $x$ in a neighborhood of the endpoints $\pm 1/2$, but the size of this neighborhood may depend on $\tilde\xi$.
\item \label{thm:L_properties3} For fixed $x\in [-\tilde\be, \tilde\be]$, the function $\tilde L(x; b,\tilde\xi)$ is an increasing function of $\tilde\xi$. As $\tilde\xi\to\infty$, $\tilde\beta$ converges to $1/2$, and $\tilde L(x; b,\tilde\xi)$ approaches a strictly negative constant.
\item  \label{thm:L_properties4} For any $b>2$,  $\tilde L(x; b,\tilde\xi)<0$ for all $\tilde\xi\ge 1$ and $x\in [-1/2,1/2]$.
\item  \label{thm:L_properties5} For any $1< b\le 2$, there exists a sampling density $\tilde\xi_b$ such that for any $\tilde\xi>\tilde\xi_b$,  the function $\tilde L(x; b,\tilde\xi)$ is negative for all $x\in [-1/2,1/2]$. Conversely, for each $1\le \tilde\xi<\tilde\xi_b$, there exists $x_{\tilde\xi}\in (0,1/2)$ such that $\tilde L(x; b,\tilde\xi)$ is positive for all $x\in[-1/2, -x_{\tilde\xi})\cup (x_{\tilde\xi}, 1/2]$.
\end{enumerate}
\end{prop}
\begin{rem}\label{rem1}
Based on numerical computations, it seems that there is a very simple formula for the critical value $\tilde\xi_b$:
\[
\tilde\xi_b = \frac{1}{b-1}, \qquad b\in (1,2].
\]
However we are unable to prove this formula analytically.
\end{rem}

 We now present the asymptotic formula for  $\varphi^{N}_M(e^{\frac{2\pi i}{b}x})$ outside of the interval $(-\tilde\beta,\tilde\beta)$. For each of these results we again make the technical assumptions that $N$ is odd, and $Nb$ is an integer. Once again, these assumptions are purely technical to ease the analysis.

\begin{thm}\label{upper}
Let $N$ and $M$ approach infinity in such a way that $N$ is odd and the ratio $N/M=\tilde\xi\ge 1$ is bounded. Also let $b>1$ such that $Nb\in \Z$, let $\tilde\be\in(0,\frac{1}{2})$ be defined by \eqref{def:tbeta}, and 
fix $x$ such that $\tilde\be<x<1/2$. Then the orthonormal polynomial $\varphi^{N}_M(e^{\frac{2\pi i}{b}x})$ satisfies
\begin{equation}\label{mr4ls}
\varphi^{N}_M(e^{\pm \frac{2\pi i}{b}x})= e^{M (\tilde{L}(x)-\tilde{L}(\tilde\beta))}\bigg[F(x)\sin\left(\pi N x\right)+\bigO(e^{-cM})\bigg]\bigg[1+\bigO(M^{-1})\bigg],
\end{equation}
where $c>0$ and $F(x)$ is a non-vanishing (complex) bounded analytic function of $x$ depending on both $b$ and $N/M$ (see Proposition \ref{thmasym2}). The error terms are uniform on compact subsets of $(\tilde\be,\frac{1}{2})$.
\end{thm}

This theorem implies that the orthonormal polynomials $\varphi^{N}_M(e^{ \frac{2\pi i}{b}x})$ have zeroes which are exponentially close to those sample points $x_j=\frac{j}{N}-\frac{1}{2}-\frac{1}{2N}, \quad j=1,2,\ldots, N$, which are outside the interval $[-\tilde\beta, \tilde\beta]$. Using \eqref{ubq},  this implies that the error terms $B_{N,M}^k(x)$ are well controlled close to these sample points. However, if $x$ is not close to the sample points, say half-way between two sample points, then the orthonormal polynomials become exponentially large due to the exponential factor $e^{M (\tilde{L}(x)-\tilde L(\tilde\beta))}$ and the fact that $\tilde{L}(x)-\tilde L(\tilde\beta)$ is strictly positive and in fact increasing on the interval $x\in (\tilde\be, 1/2)$. This suggests that $\left\lvert B_{N,M}^k(x)\right\rvert$ could become large between sample points. To see that this is indeed the case, we need a finer estimate on the error terms using the monic orthogonal polynomials.

To obtain a more precise estimate we analyze the quantities $B^{k}_{N,M}(x)$ more directly. Note that \eqref{def:Bk} can be written in terms of the Christoffel--Darboux kernel \eqref{km} as
\begin{equation}\label{BCD}
B^{k}_{N,M}(x) = e^{2\pi i x(k-M_0)/b}- \frac{1}{N} e^{-2\pi i M_0 x/b}\sum_{j=1}^N K_M(x, x_j)e^{2\pi i k x_j/b}.
\end{equation}
\begin{thm}\label{lower}
Fix $l\in \Z_+$, and let $k = \frac{M-1}{2} + l$, assuming $M$ is odd. For large enough $M$, there are  $c_l>0$ and $d_l>0$ independent of $M$ such that
\begin{equation}\label{resga}
c_l M^l  e^{M \tilde{L}(\tilde\beta)} \left\lvert\varphi^{N}_M(e^{\frac{2\pi i}{b}x})\right\rvert \le\left\lvert B^{k}_{N,M}(x) \right\rvert \le d_l M^l  e^{M \tilde{L}(\tilde\beta)}\left\lvert \varphi^{N}_M(e^{\frac{2\pi i}{b}x})\right\rvert.
\end{equation}
\end{thm}
Combining Theorems \ref{upper} and \ref{lower}, we obtain the following corollary which states that $\left\lvert B^{k}_{N,M}(x) \right\rvert$ may be exponentially large in $M$ when $|x|>\tilde{\beta}$ and $x$ is not near one of the sample points. This is one of the primary results of this paper. 
\begin{cor}\label{maincor}
Fix $l\in \Z_+$, let $k = \frac{M-1}{2} + l$, assuming $M$ is odd, and let $b\in(1,2]$ such that $Nb\in \N$. Fix $\epsilon>0$ and define the $N$-dependent set 
\begin{equation}
\Xi_N:= \{ x\in [-1/2, 1/2] : |x-x_j|>\epsilon/N \ \textrm{for all of the sample points} \ x_j\}.
\end{equation}
Then as $M,N\to\infty$ such that $N/M = \tilde{\xi} < \tilde\xi_b$, the quantity $\left\lvert B^{k}_{N,M}(x) \right\rvert$ is exponentially increasing in $M$ for $x$ in the set
\begin{equation}
\Xi_N \cap \{ |x|>\tilde\beta\}.
\end{equation}

\end{cor}

 
Note that Theorem \ref{lower} gives information about $B_{N,M}^k(x)$ only for $k=(M-1)/2+l$ with $l$ fixed as $M\to\infty$. Ideally one would like to obtain estimates on $B_{N,M}^k(x)$ for all $k\in \Z$, but it is unfortunately beyond the scope of the current paper. Nevertheless, it is interesting to see that for certain values of $b$ and $\tilde\xi$, $\left\lvert B^{k}_{N,M}(x) \right\rvert$ is exponentially large in $M$ for certain intervals of $x$-values close to the end-points of the interval $[-1/2,1/2]$. Of course even if the terms $\left\lvert B^{k}_{N,M}(x) \right\rvert$ are exponentially large in $M$, the sum \eqref{Error_series} could still be small due to exponential decay of the Fourier coefficients $a_k$ or cancellation of terms, but Theorem  \ref{lower}  and parts \ref{thm:L_properties4} and \ref{thm:L_properties5} of Proposition \ref{thm:L_properties} indicate that it may be desirable to choose $b>2$ or $\tilde\xi > \tilde\xi_b$ if $b\in (1,2)$. We emphasize that even this choice of parameters does not {\it guarantee} a small error, since we are unable to estimate $|B_{N,M}^k(x)|$ for $k$ on the order of $aM$ with $a>1$.  It does agree with numerical evidence, which suggests that $b=2$ is generally sufficient for exponential convergence, see e.g., \cite{Adcock-Ruan14, brhapo}.

Theorems \ref{main_band} and \ref{upper} imply that the error is small throughout the interval $(-\tilde\beta, \tilde\beta)$ as well as near the sample points $x_j=\frac{j}{N}-\frac{1}{2}-\frac{1}{2N}, \quad j=1,2,\ldots, N$ which are outside the interval $[-\tilde\beta, \tilde\beta]$, while it may be large outside of $(-\tilde\beta, \tilde\beta)$ when $x$ is between two sample points.  In the language of discrete orthogonal polynomials \cite{BKMM}, the interval $(-\tilde\beta, \tilde\beta)$ is called a {\it band}, and the intervals $[-1/2, -\tilde\beta)$ and $(\tilde\beta, 1/2]$ are called {\it saturated regions}. Below we explain this terminology and describe how Theorems \ref{main_band} and \ref{upper} may be generalized to non-uniform samplings.

\subsection{General orthogonal polynomial theory and heuristics}\label{heuristics}

Denote the set of equispaced sample points as $L_N$, so
 \begin{equation}\label{def:L_N}
L_N:=\left\{\frac{j}{N}-\frac{1}{2}-\frac{1}{2N}, \quad j=1,2,\ldots, N\right\}.
\end{equation}
In the variable $z=e^{2\pi i x/b}$ the orthogonal polynomial $\varphi_M^N(z)$ defined in \eqref{def:orthoprod} may be written as the discrete Heine formula \cite[Theorem 1.513]{Simon-OPUC-1}\footnote{In fact \cite[Theorem 1.513]{Simon-OPUC-1} gives a formula for the $M$-th orthogonal polynomial as a Toeplitz-like determinant. It is a simple exercise to go from this determinant to the multiple sum in \eqref{eq:Heine1}. See, e.g., \cite[Proposition 3.8]{Deift99} for a similar computation for orthogonal polynomials on the real line.}
\begin{equation}\label{eq:Heine1}
\varphi_M^N(z) = \frac{1}{D_{M,N}} \sum_{x_1, \dots x_M\in L_N} \prod_{j=1}^M(z-e^{2\pi ix_j/b}) \prod_{1\le j<k\le M} {\left\lvert e^{2\pi ix_k/b}-e^{2\pi ix_j/b}\right\rvert}^2,
\end{equation}
where $D_{M,N}$ is a constant which ensures that $\langle \varphi_M^N, \varphi_M^N \rangle_N = 1$. Recall that $M\le N$, and note that the multiple sum in \eqref{eq:Heine1} is over the set of all $M$-tuples of sample points in $L_N$. The normalizing constant $D_{M,N}$ is given as
\begin{equation}\label{D_multi_sum}
\begin{aligned}
\left(D_{M,N}\right)^2:=&\sum_{x\in L_N} \sum_{y_1, \dots y_M\in L_N} \sum_{x_1, \dots x_M\in L_N} \prod_{j=1}^M(e^{2\pi i x/b}-e^{2\pi ix_j/b})(e^{-2\pi i x/b}-e^{2\pi iy_j/b}) \\
&\times  \prod_{1\le j<k\le M} {\lvert e^{2\pi ix_k/b}-e^{2\pi ix_j/b}\rvert}^2{\lvert e^{2\pi iy_k/b}-e^{2\pi iy_j/b}\rvert}^2,
 \end{aligned}
 \end{equation}
which ensures that $\langle \varphi_M^N, \varphi_M^N \rangle_N = 1$. If we denote by $\nu_{\bf x}$ the normalized counting measure on the points $x_1, \dots, x_M$, 
\begin{equation}\label{eq:counting_measure}
\nu_{\bf x} = \frac{1}{M} \sum_{j=1}^M \de_{x_j},
\end{equation}
then the above integral can be written as
\begin{equation}\label{eq:Heine2}
\varphi_M^N(z) = \frac{1}{D_{M,N}} \sum_{x_1, \dots , x_M\in L_N} \exp\left[M \int \log(z-e^{2\pi iy/b})d\nu_{\bf x}(y)\right]\exp\left[-M^2 \tilde H(\nu_{\bf x})\right],
\end{equation}
where $\tilde H(\nu)$ is the functional
\begin{equation}\label{eq10a}
\tilde H(\nu)=\iint_{x\not=y} \log \frac{1}{|e^{2\pi ix/b}-e^{2\pi iy/b}|}d\nu (x)d\nu (y).
\end{equation}
Since there is a factor $M^2$ in the exponent, we expect the primary contribution in this sum as $M\to\infty$ to come from a minimizer of the functional $\tilde H(\nu)$. 
If we consider a regime in which $M, N\to\infty$, and the ratio $N/M$ remains bounded, then we find that for large $M$, the measures $\nu_{\bf x}$ converge to probability measures absolutely continuous with respect to Lebesgue measure with density not exceeding the ratio $N/M$.
Thus we minimize over all Borel measures $\nu$ on $[-1/2, 1/2]$ satisfying the following two properties:
\begin{enumerate}
\item The measure $\nu$ is a probability measure, i.e. $\int_{-1/2}^{1/2} d\nu(x) =1$.
\item The measure $\nu$ does not exceed the limiting density of nodes $x_1, \dots, x_N$ as $N, M\to\infty$. That is, $0\le \nu \le \sg \tilde\xi$, where $\sg$ is the Lebesgue measure and $\tilde\xi := \frac{N}{M}$.
\end{enumerate}
As noted in the Introduction, the problem of an equilibrium measure with constraint was studied by Rakhmanov \cite{Rakhmanov96} for a special case and a general study was done by Dragnev and Saff \cite{Dragnev-Saff97}.
Such a minimizer exists and is unique and we refer to it as the {\it equilibrium measure}, denoted $\tilde\nu_{\eq}$.
Then the formula \eqref{eq:Heine2} indicates heuristically that for large $M$,
\begin{equation}\label{eq14}
\varphi_{M}^N(z) \sim \frac{e^{-M^2 E_0}}{D_{M,N}}\exp\left(M\int_{-1/2}^{1/2} \log(z-e^{2\pi iy/b})d\tilde\nu_{\eq}(y)\right),
\end{equation}
where $E_0:=H(\tilde\nu_{\eq})$.
The equilibrium measure $\tilde\nu_{\eq}$ is uniquely determined by the {\it Euler--Lagrange variational conditions}:
there exists a {\it Lagrange multiplier} $l$ such that
\begin{equation}\label{eq15a}
2\int \log|e^{2\pi ix/b}-e^{2\pi iy/b}| d\tilde\nu_{\eq} (y) \left\{
\begin{aligned}
&\geq l \quad \textrm{for}\quad x \in \supp \tilde\nu_{\eq}\\
&\leq l \quad \textrm{for}\quad x \in \supp (\tilde\xi\sg-\tilde\nu_{\eq}).
\end{aligned}\right.
\end{equation}
The support of the equilibrium measure can be divided into two pieces: one in which the upper constraint is active, $\supp \tilde\nu_{\eq}\setminus \supp (\tilde\xi\sg-\tilde\nu_{\eq})$; and one in which it is not, $\supp \tilde\nu_{\eq}\cap \supp (\tilde\xi\sg-\tilde\nu_{\eq})$. The former is referred to as the {\it saturated region} and the latter as the {\it band}. Later we show that the band is the interval $(-\tilde\beta,\tilde\beta)$ and the saturated region consists of the two intervals $(-1/2,-\tilde\beta)\cup (\tilde\beta,1/2)$, so Theorem \ref{main_band} refers to the behavior of $\varphi_{M}^N(e^{2\pi ix/b})$ in the band, and Theorem \ref{upper} refers to the behavior in the saturated region.
In the band we have
\begin{equation}\label{eq16}
2\int \log|e^{2\pi ix/b}-e^{2\pi iy/b}| d\tilde\nu_{\eq} (y)=l
  \quad \textrm{for}\quad x \in \supp \tilde\nu_{\eq}\cap \supp (\tilde\xi\sg-\tilde\nu_{\eq}),
\end{equation}
and for $x \in \supp \tilde\nu_{\eq}\cap \supp (\tilde\xi\sg-\tilde\nu_{\eq})$ we have that
\begin{equation}\label{eq16aa}
\varphi_{M}^N(e^{2\pi i x/b}) \sim \frac{e^{-M^2 E_0+Ml/2}}{D_{M,N}}.
\end{equation}
A similar heuristic argument starting from \eqref{D_multi_sum} indicates that 
\begin{equation}\label{eq16b}
D_{M,N}\sim e^{M^2 E_0-Ml/2},
\end{equation}
thus $\varphi_{M}^N(e^{2\pi i x/b}) =\bigO(1)$ for $x \in \supp \tilde\nu_{\eq}\cap \supp (\tilde\xi\sg-\tilde\nu_{\eq})$.

On the other hand, for $x$ in the saturated regions, 
\begin{equation}\label{eq16a}
2\int \log|e^{2\pi ix/b}-e^{2\pi iy/b}| d\tilde\nu_{\eq} (y)\ge l
  \quad \textrm{for}\quad x \in \supp \tilde\nu_{\eq} \setminus \supp (\tilde\xi\sg-\tilde\nu_{\eq}).
\end{equation}
We will show in Section \ref{L_properties} that this inequality is in fact strict in the intervals $[-1/2, -\tilde\beta) \cup (\tilde\beta, 1/2]$, so \eqref{eq14}, \eqref{eq16aa}, and \eqref{eq16b} imply that 
\begin{equation}\label{def:tildeL}
\varphi_{M}^N(e^{2\pi ix/b}) \sim e^{M(\tilde L(x)-l/2)}, \quad \tilde L(x) =  \int_{-1/2}^{1/2} \log|e^{2\pi ix/b}-e^{2\pi iy/b}|\, d\tilde\nu_{\eq} (y)  >l/2.
\end{equation}
Thus $\left\lvert\varphi_{M}^N(e^{2\pi ix/b})\right\rvert$ is exponentially large in $M$ for $x\in [-1/2, \tilde\beta) \cup (\tilde\beta, 1/2]$. But $\langle \varphi_{M}^N, \varphi_{M}^N\rangle = 1$, so $\left\lvert \varphi_{M}^N(e^{2\pi ix_j/b})\right\rvert$ cannot be large when $x_j\in L_N$. We conclude then that $\left\lvert\varphi_{M}^N(e^{2\pi ix_j/b})\right\rvert$ oscillates very regularly in the saturated region, nearly vanishing at each node of $L_N$, and then growing exponentially large between nodes, as in Theorem \ref{upper}. This is indeed the meaning of the term saturated region.  A well known property of polynomials  on the unit circle is that all their zeros lie strictly
inside the circle, and the polynomials \eqref{def:orthoprod} are in a class whose zeroes approach the circle as $M\to\infty$, but the discrete measure constrains the number of zeroes which can approach any mass point to at most one, see \cite[Theorem 1.7.20]{Simon-OPUC-1}. Thus the zeroes are saturated to the maximal density allowed by the discrete measure.

\subsection{The effect of the sampling density}

From the previous subsection, we find that the orthonormal polynomial $\left\lvert \varphi_{M}^N(e^{2\pi ix/b})\right\rvert$ oscillates with exponentially large amplitude for $x$ in the saturated region $(-1/2,-\tilde\beta)\cup(\tilde\beta,1/2)$, and is order 1 in the band $(-\tilde\beta,\tilde\beta)$. We note here the similarity with a result of Rakhmanov \cite[Theorem 1]{Rakhmanov07}, who showed that any polynomial of degree $M$ with unit discrete norm $|| \cdot ||_N$ is necessarily uniformly bounded in the interval $[-r,r]$, where 
\begin{equation}\label{def:band_r}
r:=\frac{ \sqrt{1-M^2/N^2}}{2}.
\end{equation}
This result does not directly apply to our case since we are dealing with trigonometric polynomials, but the similarity in the results is striking.
Since $\left\lvert \varphi_M^N(e^{2\pi i xb})\right\rvert$ is only large in the saturated region, it may be advantageous to make the saturated region as small as possible. From \eqref{def:tbeta} we find that $\tilde\beta$ is increasing in the sampling density $\tilde\xi=N/M$ and $\tilde\beta = 1/2-\bigO(1/\tilde\xi^{2})$ as $\tilde\xi \to\infty$. Thus increasing the sampling density makes the saturated region smaller, but the saturated region exists whenever the sampling scheme is comprised of equispaced data and $N=\bigO(M)$. 

If one were to take $N$ much larger than $M$, say $N=\bigO(M^{1+\ep})$ for some $\ep>0$, then the saturated region would vanish in the limit as $M\to\infty$. However, for finite $M$, there is still a saturated region in a neighborhood of $x=\pm 1/2$ of order $\bigO(M^{-2\ep})$. This can be seen by writing $N=$const.$\cdot M^{1+\ep}$ in \eqref{def:tbeta} and solving for $\tilde \be$ as $M\to\infty$. Heuristically, this saturated region is negligible if it is smaller than the spacing between sample points, which is $\bigO(M^{1+\ep})$. This implies that the saturation phenomenon is detectable on a shrinking interval when $N=\bigO(M^{\kappa})$ for $1<\kappa<2$. If $N=\bigO(M^2)$, then the size of the saturated region is of the same order as the spacing between sample points, and therefore plays no role. It is already known that  $N=\bigO(M^2)$ is necessary and sufficient for the Fourier extension approximation to be well conditioned, see \cite{adhum, Platte-Trefethen-Kuijlaars11}. The heuristic explanation above seems to indicate a similar result for the convergence: uniform convergence of the Fourier extension approximation is guaranteed for $N=\bigO(M^2)$, but not for $N=\bigO(M^{\kappa})$ with $\kappa<2$. The sufficiency of the condition $N=\bigO(M^2)$ for fast uniform convergence is suggested by the aforementioned result of Rakhmanov \cite{Rakhmanov07}, see also \cite{Schonhage61, Coppersmith-Rivlin92, Ehlich66, Ehlich-Zeller64, Ehlich-Zeller65}, with the caveat that those papers deal with polynomials rather than trigonometric polynomials. The necessity of this condition is strongly indicated by Theorem \ref{lower} and Corollary \ref{maincor}.

Instead of taking the sample points $x_j$ to be equally spaced, one could also take them to approach some non-constant density as $N\to\infty$. Indeed, suppose the $N$ sample points are taken such that the counting measure $\frac{1}{N} \sum_{j=1}^N \de_{x_j}$ converges weakly to some density $\varrho(x)$ as $N\to\infty$. Then all of the heuristic arguments of Section \ref{heuristics} are still valid with the constant upper constraint $\tilde\xi$ replaced by the variable constraint $\tilde\xi \varrho(x)$. That is, the equilibrium measure is obtained by minimizing the functional \eqref{eq10a} over the space of Borel probability measures $\nu$ satisfying $0\le \nu \le \tilde\xi \varrho(x)\sg$, where once again $\sg$ is the Lebesgue measure and $\tilde\xi = N/M$. 

\begin{figure}\label{unconstrained_plot}
\begin{center}
\scalebox{0.35}{\includegraphics{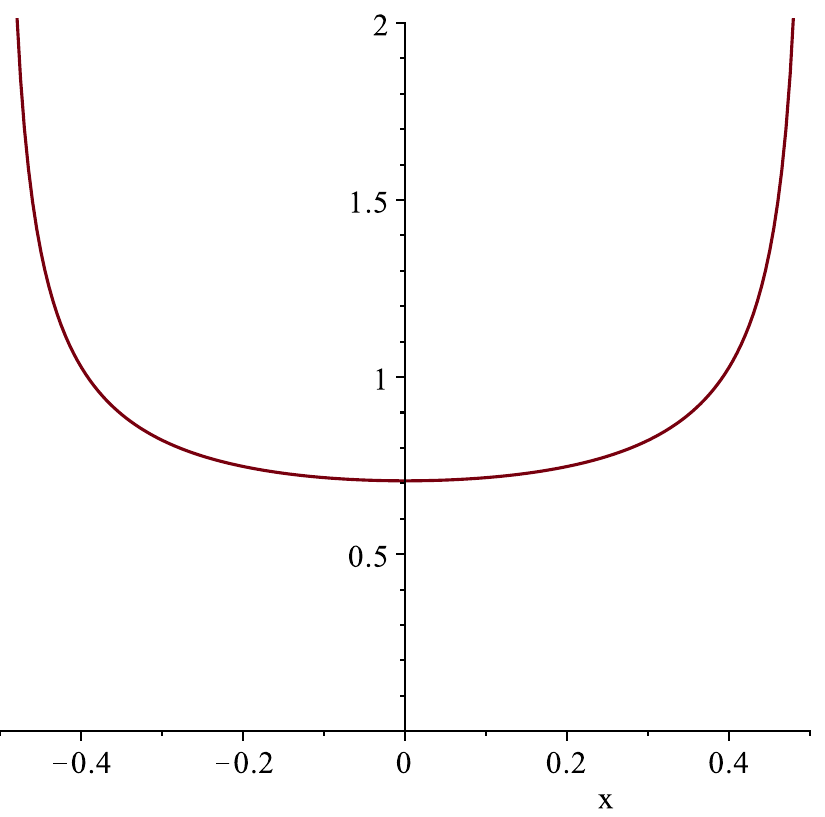}}
\caption{A plot of the unconstrained equilibrium measure for $b=2$.}
\label{unconstrained_plot}
\end{center}
\end{figure}

To determine which sampling densities $\varrho(x)$ will cause the saturated region to vanish, we can consider the unconstrained equilibrium problem, which simply minimizes \eqref{eq10a} over the space of Borel probability measures on $[-1/2,1/2]$. This is exactly the equilibrium problem which describes the asymptotic behavior of the continuous orthogonal polynomials, since the upper constraint is a manifestation of the discrete orthogonality. The unconstrained equilibrium problem can be solved explicitly, and its solution is 
\begin{equation}\label{unconstrained}
d\tilde\nu_{\eq}^c(x) = \frac{\sqrt{2}\cos(\pi x/b)}{b\sqrt{\cos(2\pi x/b)-\cos(\pi/b)}}dx.
\end{equation}
Thus in the constrained equilibrium problem, the upper constraint is only active if the sampling density $\tilde\xi \rho(x)$ is smaller than the unconstrained equilibrium density above. A plot of this density is given in Figure \ref{unconstrained_plot}. Note that this density diverges as $x\to \pm 1/2$, so any finite sampling density will produce saturated regions close the endpoints $\pm 1/2$. However if one were to take the sampling density $\tilde\xi \rho(x)$, to be exactly the unconstrained equilibrium measure, then there will be no saturated region as $M\to\infty$. Practically, for finite $M$ and $N$, this means sampling with a much higher density near the endpoints of the interval $[-1/2,1/2]$ than in the middle. Indeed this has been suggested in the literature, see \cite{adhum}. For results on the relationship between convergence rates and stability, see \cite{Platte-Trefethen-Kuijlaars11} for equispaced data, and \cite{Adcock-Platte-Shadrin18} for data which is not necessarily equispaced.

\subsection{Outline for the rest of the paper}
The plan for the rest of the paper is as follows. In Section \ref{asy_formulas} we present very precise asymptotic formulas for the monic orthogonal polynomials \eqref{def:orthoprod} which imply Theorems \ref{main_band}, \ref{upper}, and \ref{lower}. In Section \ref{L_properties} we prove the properties of $\tilde L(x)$ given in Proposition \ref{thm:L_properties}, and in Section \ref{main_tech_proof} we prove Proposition \ref{subexpgr}, which is the main technical ingredient in the proof of Theorem \ref{lower}. In Section \ref{eq_measure} we derive the main quantities necessary in the asymptotic analysis of the orthogonal polynomials  \eqref{def:orthoprod}, including the equilibrium measure and the function $\tilde{L}(x)$. Finally in Section \ref{RH_analysis} we prove the asymptotic results stated in Section \ref{asy_formulas} using the Riemann--Hilbert method.

\subsection{Acknowledgments}
Both authors are grateful to Oscar Bruno, John Boyd, Ben Adcock, and Vilmos Totik for helpful comments. KL is supported by the Simons Foundation through grant \#357872.

\section{Precise asymptotic formulas for the monic orthogonal polynomials}\label{asy_formulas}

In the results below, we make the technical assumption that $N$ is odd. Furthermore we assume that $b>1$ is chosen so that $Nb$ is an integer, and introduce the notations 
\begin{equation}\label{alb}
 \al:= \pi/b\, \qquad m:= Nb .
\end{equation}
Notice then that 
\begin{equation}
e^{im\al} = e^{iNb} = -1,
\end{equation}
so $e^{i\al}$ is exactly halfway between two $m$-th roots of unity.
Let $\T\subset \C$ be the unit circle and let $C_\al$ be the arc
\begin{equation}\label{in2}
C_\al=\left\{e^{i\theta} : -\al \le \theta \le \al \right\} \subset \T.
\end{equation}
Furthermore let
\begin{equation}\label{in1}
\lattice:=\left\{z\in C_\al : z^m=1\right\} \subset \T,
\end{equation}
be the set of $m$-th roots of unity which sit inside the arc $C_\al$. Note then that the orthogonality \eqref{def:orthoprod} can be written as
 \begin{equation}\label{def:orthoprod2}
 \frac{b}{m}\sum_{z\in \lattice} \overline{\varphi^{N}_k(z)}\varphi^{N}_\ell(z) = \de_{k\ell}, \quad z_j = e^{2\pi i x_j/b}.
 \end{equation}
In what follows it will be convenient to consider the monic versions of these orthogonal polynomials as well. That is, let $p_k(z) = z^k +\dots$ be the monic polynomial of degree exactly $k$ satisfying
 \begin{equation}\label{def:orthoprod3}
 \frac{1}{m}\sum_{z\in \lattice} \overline{p_k(z)} p_\ell(z) = h_k\de_{k\ell},
 \end{equation}
 for some sequence of positive constants $\{h_k\}_{k=0}^\infty$. These polynomials are related to the ones $\varphi^{N}_k(z)$ as
  \begin{equation}\label{phi_p_relation}
  \varphi^{N}_k(z)= \frac{p_k(z)}{\sqrt{bh_k}}.
  \end{equation}
 
Since the lattice $L_{\al, m}$ is symmetric about the real axis, the polynomials $p_k(z)$ have real coefficients. Also notice that these orthogonal polynomials only exist for $0\le k < | L_{\al, m} | = m\al/\pi=N$. As with all orthogonal polynomials on the unit circle, they satisfy the Szeg\H o recursion
\begin{equation}\label{in3a}
zp_j(z)=p_{j+1}(z)+\rho_{j+1}p_j^*(z) \quad \text{and} \quad p^*_j(z)=p^*_{j+1}(z)+\rho_{j+1}zp_j(z),
\end{equation}
where $p_j^*(z):=z^jp_j(z^{-1})$ is the reverse polynomial to $p_j(z)$, and $\rho_{j+1}=-p_{j+1}(0)$ is the Szeg\H o parameter. We have used the fact that due to the conjugate symmetry the Szeg\H o parameters $\rho_j$ are real. The normalizing constants are related to these parameters by (see e.g. \cite{Simon-OPUC-1})
 \begin{equation}\label{in3d}
  h_M = \prod_{j=0}^{M-1}\left(1-\rho_{j+1}^2\right).
  \end{equation}

 Below we state precise asymptotic formulas for the polynomials $p_M(z)$ and the normalizing constants $h_M$ as $M\to\infty$. The results are described in terms of the angle $\phi\in[-\al,\al]$ where $z=e^{i \phi}$. Note that in Section \ref{intro} we denoted $z=e^{2\pi i x/b}$ so to use these results to prove Theorems \ref{main_band}, \ref{upper}, and \ref{lower}, we must take $\phi = 2\pi x/b$.
 
 Our results are stated in terms of the equilibrium measure and Lagrange multiplier discussed in Section \ref{heuristics}. In the variable $\phi$, the equilibrium measure is defined as the unique measure on $[-\al,\al]$ which minimizes the functional
 \begin{equation}\label{eq10}
H(\nu)=\iint_{\phi\not=\theta} \log \frac{1}{|e^{i\phi}-e^{i\theta}|}d\nu (\phi)d\nu (\theta),
\end{equation}
in the space $\mcal$ of probability measures on $[-\al,\al]$, where
\begin{equation}\label{eq11}
\mcal=\left\{\nu:\;0\leq \nu\leq \frac{\xi}{2\pi}\sg,\;\;\nu[-\al,\al]=1\right\},
\end{equation}
\begin{equation}\label{in4}
\xi:= \frac{m}{M} >b>1,
\end{equation}
and $\sg$ is the Lebesgue measure. Note that $\xi = b\tilde\xi$. The Euler--Lagrange variational conditions, which determine the equilibrium measure $\nu_{\eq}$ uniquely, are
\begin{equation}\label{eq15}
2\int_{-\al}^{\al} \log|e^{i\phi}-e^{i\theta}| d\nu_{\eq} (\theta) \left\{
\begin{aligned}
&\geq l \quad \textrm{for}\quad \phi \in \supp \nu_{\eq}\\
&\leq l \quad \textrm{for}\quad \phi \in \supp \left(\frac{\xi}{2\pi}\sg-\nu_{\eq}\right),
\end{aligned}\right.
\end{equation}
where $l$ is the Lagrange multiplier.

The following propositions give formulas for the equilibrium measure and Lagrange multiplier.
\begin{prop}\label{EqMeasure}
The equilibrium measure is given by the formula
\begin{equation}\label{eq34}
d\nu_{\eq}(\th)=\rho(\th)d\th\,,
\end{equation}
where
\begin{equation}\label{eq35}
\rho(\th)=\left\{
\begin{aligned}
&\frac{\xi}{\pi^2}\arctan\left(\frac{\sqrt{2}\tan\left(\frac{\pi}{2\xi}\right)\cos(\th/2)}{\sqrt{\cos\th-\cos\be}}\right)\,, \qquad &\th\in [-\be, \be] \\
&\frac{\xi}{2\pi} \,, \qquad &\th\in [-\al, -\be] \cup [\be, \al],
\end{aligned}\right.
\end{equation}
where $\be \in (0,\al)$ is given by the equation
\begin{equation}\label{eq36}
\cos\be=\cos\al+(1+\cos\al)\tan\left(\frac{\pi}{2\xi}\right)^2.
\end{equation}
\end{prop}
Proposition \ref{EqMeasure} is proved in Section \ref{EqMeasureProof}.
\begin{prop}\label{prop:lm}
The Lagrange multiplier is given by the formula
\begin{equation}\label{eq39}
\begin{aligned}
l=2\int_{-\al}^\al \log|1-e^{i\th}|\rho(\th)\,d\th=-\frac{2\xi\tan\left(\frac{\pi}{2\xi}\right)}{\pi}\int_{1}^{\sqrt{\frac{2}{1-B}}}\frac{\log \left[\frac{1+Bx^2+x\sqrt{1+B}\sqrt{2-x^2(1-B)}}{x^2-1}\right]}{1+x^2 \tan\left(\frac{\pi}{2\xi}\right)^2}\,dx,
\end{aligned}
\end{equation}
where $B:=\cos\beta.$
In particular, $l<0$ for all $\al\in (0,\pi)$ and $\xi>1$, and
\begin{equation}\label{lm10}
\lim_{\xi\to\infty} l=-\int_{1}^{\sqrt{\frac{2}{1-A}}}\log \left[\frac{1+Ax^2+x\sqrt{1+A}\sqrt{2-x^2(1-A)}}{x^2-1}\right]\,dx<0\,,
\end{equation}
where $A:=\cos\alpha$. Note that $\lim_{\xi\to\infty} \beta = \alpha$, so $A=\cos\alpha=\cos\beta=B$ in this limit.
\end{prop}
Proposition \ref{prop:lm} is proved in Section \ref{prop_lmProof}.

We also introduce the function
\begin{equation}\label{mr1}
J(\theta):=\frac{\cos\theta-\cos\be}{1-\cos(\be-\theta)},
\end{equation}
and the functions related to the equilibrium measure
\begin{equation}\label{mr1b}
\begin{aligned}
 I(\phi):=\int_\phi^\al \rho(\th)\,d\th\,, \qquad L(\phi)\equiv L(\phi;\al,\xi)&:=\int_{-\al}^\al \log |e^{i\phi}-e^{i\th}|\rho(\th)\,d\th \\
&=\frac{1}{2}\int_{-\al}^\al \log\big(2(1-\cos(\phi-\th))\big)\rho(\th)\,d\th\,,
\end{aligned}
\end{equation}
defined for $\phi \in [-\al,\al]$; and
\begin{equation}\label{eq20}
g(z)=\int_{-\al}^\al \log(z-e^{i\th}) \rho(\th)d\th
  \quad \textrm{for}\quad z \in \C\setminus\big((-\infty,-1]\cup\T\big),
\end{equation}
where for a given $\th\in(-\al,\al)$, the function 
\begin{equation}\label{eq21}
\log(z-e^{i\th})=\log|z-e^{i\th}|+i\arg(z-e^{i\th})
\end{equation}
has the cut on the contour
\begin{equation}\label{eq22}
\Ga_{\th}=(-\infty,-1)\cup \{ z=e^{i\phi}\,|\, -\pi\le \phi\le \th\}.
\end{equation}


Note that $L(2\pi x/b) = \tilde{L}(x)$, where $\tilde L(x)$ is defined in \eqref{def:tildeL}.
We have a formula for the derivative of the $g$-function.
\begin{prop}\label{g-function}
The derivative of the $g$-function is given by the formula
\begin{equation}\label{eq37}
g'(z)=\frac{1}{2z}+\frac{\xi}{\pi z}\arctan\left(\frac{(z+1)\tan\left(\frac{\pi}{2\xi}\right)}{\sqrt{R(z)}}\right)\,,
\end{equation}
where
\begin{equation}\label{eq38}
R(z)=(z-e^{i\be})(z-e^{-i\be})=z^2-2z\cos\be+1\,,
\end{equation}
with $\be$ as defined in \eqref{eq36}. The function $\sqrt{R(z)}$ is taken with a cut on the arc $C_\be$, taking the branch such that $\sqrt{R(z)}\sim z$ as $z\to\infty$.
\end{prop}
Proposition \ref{g-function} is proved in Section \ref{EqMeasureProof}.

We can now state the asymptotic formulas for the orthogonal polynomials \eqref{def:orthoprod3} on the arc $C_\al$.
The following propositions describe the asymptotic behavior of $p_M(z)$ in the band and the saturated regions, respectively.
\begin{prop}\label{monic_band}
For  $z=e^{i\phi}$ with $ -\be < \phi < \be$, the polynomial $p_M(e^{i\phi})$ satisfies as $M\to\infty$,
\begin{equation}\label{mr2}
\begin{aligned}
p_M(e^{i\phi})=e^{\frac{M}{2}(l+i\phi+i\pi)}\bigg[e^{-i\be/4} J(\phi)^{1/4}\cos(M\pi I(\phi)-\pi/4)-e^{i\be/4} &J(\phi)^{-1/4}\sin(M\pi I(\phi)-\pi/4) \\
& \qquad+\bigO(M^{-1})\bigg]\,.
\end{aligned}
\end{equation}
The error term is uniform on compact subsets of $ \{\phi :-\be < \phi < \be\}$.
\end{prop}
\begin{rem}
If we consider the regime $\al=\pi$ and $\xi\to\infty$, then these polynomials become the continuous orthogonal polynomials on $\T$ with uniform weight, and $p_M(z)=z^M$. In this case the quantities in the above proposition become
\begin{equation}\label{mr3}
l=0, \quad \be=\pi, \quad J(\phi)\equiv 1, \quad I(\phi)=\frac{\pi-\phi}{2\pi}\,,
\end{equation}
and it is straightforward to see that the formula \eqref{mr2} reduces to $e^{iM\phi}$ as expected.
\end{rem}

\begin{prop}\label{thmasym2}
For  $z=e^{i\phi}$ with $\be < \phi < \al $, the polynomial $p_M(e^{\pm i\phi})$ satisfies as $M\to\infty$,
\begin{equation}\label{mr4}
\begin{aligned}
p_M(e^{\pm i\phi})&=\frac{e^{ML(\phi)}e^{-\frac{i\phi}{2}(m-M)}e^{\pm iM\pi/2}e^{\pm im\al/2}}{2} \\
&\qquad \times \bigg[\big(e^{-i\be/4}(-J(\phi))^{1/4}+e^{i\be/4}(-J(\phi))^{-1/4}\big)\big(1-z^m\big)+\bigO(e^{-cM})\bigg] \bigg[1+\bigO(M^{-1})\bigg]\,,
\end{aligned}
\end{equation}
for some constant $c>0$. The error terms are uniform on compact subsets of $\{\phi :  \be < \phi < \al\} $.
\end{prop}

In order to obtain an expression which is uniform all the way up to the endpoints $e^{\pm i\al}$, we must introduce the following function:
\begin{equation}\label{mr5}
\tilde{D}_{\pm \al}(\phi)=
 \frac{\sqrt{2 \pi}\left(\frac{m(\al \mp \phi)}{2\pi}\right)^{\frac{m(\al\mp \phi)}{2\pi}-1}}{\Ga\left(\frac{m(\al\mp \phi)}{2\pi}-\frac{1}{2}\right)e^{\frac{m(\al\mp\phi)}{2\pi}}}\,.
\end{equation}
According to Stirling's formula, $\tilde{D}_{\pm \al}(\phi)=1+\bigO(m^{-1})$, whenever $\pm \phi<\al-\ep$. We then have the following formula for $p_M(z)$ when $z$ is close to the endpoints $e^{\pm i\al}$.
\begin{prop}\label{hard_edge}
There exists $\ep>0$ such that for all $\phi \in (\al-\ep, \al]$, the polynomial $p_M(e^{\pm i\phi})$ satisfies as $M\to\infty$,
\begin{equation}\label{mr6}
\begin{aligned}
p_M(e^{\pm i\phi})&=\frac{e^{ML(\phi)}e^{-\frac{i\phi}{2}(m-M)}e^{\pm iM\pi/2}e^{\pm im\al/2}}{2} \\
&\qquad \times\bigg[\frac{e^{-i\be/4}(-J(\phi))^{1/4}+e^{i\be/4}(-J(\phi))^{-1/4}}{\tilde{D}_{\pm\al}(\phi)}\big(1-z^m\big)+\bigO(e^{-cM})\bigg] \bigg[1+\bigO(M^{-1})\bigg]\,,
\end{aligned}
\end{equation}
for some constant $c>0$. The errors are uniform on the specified interval.
\end{prop}

We now describe the asymptotics of $p_M(z)$ close to the point $e^{i\be}$. For any $z\in \C$ we denote the disc of radius $\ep$ around $z$ by $D(z,\ep)$. 
Introduce the function
\begin{equation}
\begin{aligned}
\psi(z):= - \left[\frac{3\pi}{2} \int_\phi^\be \left(\frac{\xi}{2\pi}-\rho(\theta)\right)\,d\theta\right]^{2/3}, \quad z=e^{i\phi}\in C_\be \cap D(e^{i\be},\ep), \\
\end{aligned}
\end{equation}
which is a priori defined for $z\in C_\be$, but extends to an analytic function in the disc $D(e^{i\be},\ep)$. Also introduce the function 
\begin{equation}
\ga(z):=\left(\frac{z-e^{-i\be}}{z-e^{i\be}}\right)^{1/4},
\end{equation}
taking the cut on $C_\be$, and the branch which is $1$ at $\infty$. The functions $\Ai$ and $\Bi$ are the usual Airy functions \cite{Olver74}. We have the following theorem.

\begin{prop}\label{turning_points}
There exists $\ep>0$ such that for all $z\in D(e^{i\be},\ep)$, the polynomial $p_M(z)$ satisfies as $M\to\infty$,
\begin{equation}\label{mr7}
\begin{aligned}
p_M(z)&=-i\sqrt{\pi}e^{\frac{i(m\al+M\pi)}{2}}e^{\frac{Ml}{2}}z^{M/2}\bigg[M^{1/6} \psi(z)^{1/4}\ga(z)\bigg(\Ai\left(M^{2/3}\psi(z)\right)\left(\frac{z^{m/2}+z^{-m/2}}{2}\right)\\
&\hspace{5cm}+\Bi\left(M^{2/3}\psi\big(z\big)\right)\left(\frac{z^{m/2}-z^{-m/2}}{2i}\right)+\bigO(M^{-1})\bigg) \\
&\qquad+\frac{1}{M^{1/6} \psi(z)^{1/4}\ga(z)}\bigg(\Ai'\left(M^{2/3}\psi(z)\right)\left(\frac{z^{m/2}+z^{-m/2}}{2}\right) \\
&\hspace{5cm}+\Bi'\left(M^{2/3}\psi\big(z\big)\right)\left(\frac{z^{m/2}-z^{-m/2}}{2i}\right)+\bigO(M^{-1})\bigg)\bigg]. \\
\end{aligned}
\end{equation}
The errors are uniform on the disc of radius $\ep$.
\end{prop}

For $z\in \C$ bounded away from the arc $C_\al$, we have the following asymptotics:
\begin{prop}\label{rho_M}
For $z\in \C$ bounded away from the arc $C_\al$, as $n\to \infty$, 
\begin{equation}\label{mr8}
\begin{aligned}
p_M(z)&=\frac{e^{Mg(z)}}{2}\left(\ga(z)+\ga(z)^{-1}\right)\left(1+\bigO(M^{-1})\right).
\end{aligned}
\end{equation}
The error is uniform on compact subsets of $\C \setminus C_\al$.
Plugging in $z=0$, we get the following asymptotic formula for the Szeg\H{o} parameters:
\begin{equation}\label{mr9}
\begin{aligned}
\rho_M=-p_M(0)&=(-1)^M\cos(\be/2)\left(1+\bigO(M^{-1})\right).
\end{aligned}
\end{equation}
\end{prop}

Finally, we give the asymptotic formula for the normalizing constant $h_M$.
\begin{prop}\label{asymhn}
As $M\to \infty$, the normalizing constants $h_M$ satisfy
\begin{equation}\label{mrlc}
h_M=\frac{e^{Ml}e^l}{\sin(\be/2)}\left(1+\bigO(M^{-1})\right).
\end{equation}
\end{prop}

For estimates involving $\left\lvert B_{N,M}^k(x)\right\rvert$, let
\begin{align}\label{chrdaym}
K_{n,M}(w)&=\frac{1}{m}\sum_{j=0}^{M-1}\sum_{z \in L_{\al,m}}\frac{1}{h_j}\overline{p_j(z)}p_j(w)z^n\nonumber\\&=\frac{1}{m}\sum_{z\in
L_{\al,m}}z^{n-M}\frac{1}{h_M}\frac{p_M(z)p_M^*(w)-p^*_M(z)p_M(w)}{1-\bar z w},
\end{align}
where in the second line we have used the Christoffel--Darboux formula, see e.g. \cite{Simon-OPUC-1}.
Using \eqref{BCD} along with \eqref{alb} and \eqref{phi_p_relation}, we immediately see that
\begin{equation}
B_{N,M}^k(x) =w^{k-M_0} - K_{k-M_0,M}(w)\,, \qquad w\equiv e^{2\pi ix/b}.
\end{equation}

We have:
\begin{prop}\label{subexpgr}
For $k$ fixed $1\le k\le M\le m\alpha/2\pi$ and $n=M+k<m$ there is a constant $d_k>0$ independent of $M$ such that
\begin{equation}\label{resgu}
|K_{M+k,M}(w)-w^{M+k}|\le d_k M^k |p^*_M(w)|.
\end{equation}
For $k$ fixed $k\le M\le m\alpha/2\pi$, and $i_0<M$, $i_0$ given in Lemma~\ref{lemc} and $n=M+k<m$ there is a constant $f_k>0$ such that
\begin{equation}\label{resgl}
f_k M^k|p^*_M(w)|\le|K_{M+k,M}(w)-w^{M+k}|.
\end{equation}
\end{prop}

With the above propositions we can now prove Theorems \ref{main_band}, \ref{upper}, and \ref{lower}.
\begin{proof}[Proof of Theorems \ref{main_band}, \ref{upper}, and \ref{lower}.]
We recall the relation \eqref{phi_p_relation} and the change of variable $\phi = 2\pi x/b$. Also note that the Lagrange multiplier $l$ is equal to $2L(\be)$ by \eqref{eq15} and \eqref{mr1b}.  Then Propositions \ref{monic_band} and \ref{asymhn} immediately imply Theorem \ref{main_band}, and Propositions \ref{thmasym2} and \ref{asymhn} immediately imply Theorem \ref{upper} with $\tilde L(x) := L(2\pi x/b)$. Also combining Propositions \ref{monic_band}, \ref{thmasym2}, \ref{hard_edge}, with \ref{turning_points} gives that Proposition \ref{subexpgr} immediately implies Theorem \ref{lower}, where we use the facts that $|p_M(y)| = |p_M^*(y)|$ and $L(\phi) \equiv l/2$ for $-\be \le \phi \le \be$.
\end{proof}

%

\section{Proof of Proposition \ref{thm:L_properties}}\label{L_properties}

The properties listed in Proposition \ref{thm:L_properties} are implied by from the following lemmas on the function $L(\phi; \al,\xi)$ as defined in \eqref{mr1b}.

\begin{lem}\label{Lphiin}
$L(\phi;\alpha,\xi)$ is an increasing function of $\phi$ for $\beta\le\phi\le\alpha$.
\end{lem}

\begin{proof} We show that $\frac{\partial}{\partial\phi}L(\phi;\alpha,\xi)>0$. Differentiating with respect to $\phi$
yields
\begin{align}\label{eq4}
\frac{\partial}{\partial\phi}L(\phi,\alpha,\xi)&=\frac\xi{4\pi}\int^{-\beta}_{-\alpha}
\frac{\sin (\phi-\theta)}{1-\cos(\phi-\theta)}\,d\theta +
\frac\xi{4\pi}\int^\alpha_{\beta}
\frac{\sin (\phi-\theta)}{1-\cos(\phi-\theta)}\,
d\theta\nonumber\\
&\quad +\frac12\int^\beta_{-\beta}
\frac{\sin(\phi-\theta)}{1-\cos(\phi-\theta)}\, \rho(\theta)\,d\theta
\nonumber\\
&=\frac\xi{4\pi}\int^\alpha_\beta
\frac{\sin(\phi+\theta)}{1-\cos(\phi+\theta)}\, d\theta+
\frac\xi{4\pi}\int^\alpha_\beta
\frac{\sin (\phi-\theta)}{1-\cos(\phi-\theta)}\,d\theta\nonumber\\
&\quad +\frac12\int^\beta_{-\beta}\frac{\sin(\phi-\theta)}{1-\cos(\phi-\theta)}\,
\rho(\theta)d\theta\nonumber\\
&=\frac\xi{4\pi}
\log\left(\frac{1-\cos(\phi+\alpha)}{1-\cos(\phi+\beta)}\right)+\frac\xi{4\pi}\log\left(\frac{1-\cos(\phi-\beta)}
{1-\cos(\phi-\alpha)}\right)\nonumber\\
&\quad +\frac12\int^\beta_{-\beta}\frac{\sin(\phi-\theta)}{1-\cos(\phi-\theta)}\,
\rho(\theta)d\theta\nonumber\\
&=\frac\xi{4\pi}\log
\frac{1-\cos(\phi+\alpha)}{1-\cos(\phi-\alpha)}+\frac\xi{4\pi}\log
\left(\frac{1-\cos(\phi-\beta)}{1-\cos(\phi+\beta)}\right)\nonumber\\
&\quad + \frac12\int^\beta_{-\beta}
\frac{\sin(\phi-\theta)}{1-\cos(\phi-\theta)}\,\rho(\theta)d\theta,
\end{align}
which equals
\begin{align}\label{eq5}
&=\frac\xi{4\pi}\int_{-\phi}^{\phi}\frac{\sin(\alpha-\theta)}{1-\cos(\alpha-\theta)}\tilde\rho(\theta)\,d\theta+\frac\xi{4\pi}\int_{-\phi}^{\phi}\frac{\sin(\alpha-\theta)}{1-\cos(\alpha-\theta)}(1-\tilde\rho(\theta))\,d\theta\nonumber\\&\quad+ \frac\xi{4\pi}\int^\beta_{-\beta}
\frac{\sin(\phi-\theta)}{1-\cos(\phi-\theta)}
(\hat\rho (\theta)-1)\,d\theta,
\end{align}
where
\beq
\hat\rho(\theta)=\frac2\pi\arctan\left(
\frac{\sqrt 2\tan \left(\frac\pi{2\xi}\right)\cos\frac\theta2}
{\sqrt{\cos\theta-\cos\beta}}\right),
\eeq
and $\tilde\rho(\theta)$ is the same as $\hat\rho(\theta)$ except $\cos\beta$ is replaced by $\cos\phi$. We now split the second integral in the above equation into an integral from $-\phi$ to $-\beta$, an integral from $-\beta$ to $\beta$, and a third integral from $\beta$ to $\phi$. Then we can use the fact that $\tilde\rho$ is a decreasing function of $\cos\phi$ while $\frac{\sin(\alpha-\theta)}{1-\cos(\alpha-\theta)}$ is a decreasing function of $\alpha$ to show that
\begin{align}
\frac{\partial}{\partial\phi}L(\phi;\alpha,\xi)&\ge\frac\xi{4\pi}\left(\int_{-\phi}^{\phi}\frac{\sin(\alpha-\theta)}{1-\cos(\alpha-\theta)}\tilde\rho(\theta)\,d\theta\nonumber\right.\\
&\quad\left.+\int_{\beta}^{\phi}\frac{\sin(\alpha-\theta)}{1-\cos(\alpha-\theta)}(1-\tilde\rho(\theta))\,d\theta+\int_{-\phi}^{-\beta}\frac{\sin(\alpha-\theta)}{1-\cos(\alpha-\theta)}(1-\tilde\rho(\theta))\,d\theta\right)\nonumber\\
&=\frac\xi{2\pi}\int_0^\phi\frac{\sin{\alpha}(\cos{\theta}-\cos{\alpha})}{(1-\cos(\alpha-\theta))(1-\cos(\alpha+\theta))}\tilde
\rho(\theta)d\theta\nonumber\\&+\frac\xi{2\pi}\int_{\beta}^\phi\frac{\sin{\alpha}(\cos{\theta}-\cos{\alpha})}{(1-\cos(\alpha-\theta))(1-\cos(\alpha+\theta))}(1-\tilde\rho(\theta))d\theta\ge0\nonumber,
\end{align}
since $\tilde\rho$ is a symmetric function of $\theta$. This proves the result.

\end{proof}

\begin{lem}\label{Lphi}
\begin{enumerate}[label=(\alph*)]
\item \label{Lphia}$L(\phi;\alpha,\xi)$ is an increasing function of $\xi$ for $-\be\le\phi\le\be$.
\item \label{Lphib}$L(\alpha;\alpha,\xi)$ is a decreasing function of $\xi$. 
\end{enumerate}
\end{lem}
\begin{proof}
We have
\begin{align}
\frac{d}{d\xi} L&(\phi;\alpha,\xi)=\frac{1}{\xi} L(\phi;\alpha,\xi)\nonumber\\&-\frac{1}{2\sqrt{2}\pi\xi}\int_{-\beta}^{\beta}\ln(2(1-\cos(\phi-\theta))\frac{\sec^2(\frac{\pi}{2\xi})\cos(\frac{\theta}{2})(\cos\theta-\cos\alpha)d\theta}{\sqrt{\cos\theta-\cos\beta}(2\tan^2(\frac{\pi}{2\xi})\cos^2(\frac{\theta}{2})+\cos\theta-\cos\beta)}.
\end{align}
The last integral  simplifies to 
\beq\label{Lphi1}
\frac{1}{2\sqrt{2}\pi\xi}\int_{-\beta}^{\beta}\ln(2(1-\cos(\phi-\theta))\frac{\cos(\frac{\theta}{2})}{\sqrt{\cos\theta-\cos\beta}}d\theta.
\eeq
Introduce the notation
\beq
d\nu_\eq^\alpha(\theta) := \lim_{\xi\to\infty} \rho(\theta)\,d\theta.
\eeq
Up to a change of variable $\theta=2\pi x/b$, this is exactly the unconstrained equilibrium measure defined in \eqref{unconstrained}. Then \eqref{Lphi1} can be written as
\beq
\frac{1}{\xi}L_\be(\phi; \be)\, ; \qquad L_\be(\phi; \be):= \int_{-\beta}^{\beta}\ln|e^{i\phi}-e^{i\theta}|d\nu_{\eq}^{\beta}(\theta),
\eeq
i.e., it is $1/\xi$ times the logarithmic transform of the unconstrained equilibrium measure on the interval $[-\be,\be]$. Therefore we have that
\beq\label{eq:diff_L}
\frac{d}{d\xi} L(\phi;\alpha,\xi)= \frac{1}{\xi}\left(L(\phi;\alpha,\xi)- L_\be(\phi; \be)\right).
\eeq
Consider first the case that $-\be\le \phi\le \be$. Recall that $L(\phi;\alpha,\xi)$ is constant on this interval, and is equal to $l/2$, where $l$ is the Lagrange multiplier \eqref{eq39}. Similarly, $L_\be(\phi; \be)$ is constant on this interval and is given as $l_\be/2$, where $l_\be$ is given by the formula \eqref{lm10} with $\be$ replacing $\al$. 

Part \ref{Lphia} of the lemma is then proven provided that $l>l_\be$. This follows from general potential theoretic considerations, see \cite[Chapter II, Theorem 4.4]{Saff-Totik97}, but since we have explicit formulas for $l$ and $l_\be$, we can compute the difference directly.
Using \eqref{eq39} and \eqref{lm10}, we have
\beq\label{eq:diff_l}
l-l_\be =\int_1^{\sqrt{\frac{2}{1-B}}} \log\left[\frac{1+Bx^2+x\sqrt{1+B}\sqrt{2-x^2(1-B)}}{x^2-1}\right]\left(1-\frac{2\xi\tan\left(\frac{\pi}{2\xi}\right)}{\pi\left(1+x^2\tan\left(\frac{\pi}{2\xi}\right)^2\right)}\right)\,dx.
\eeq
Since the logarithm in the integrand is clearly positive, and
\beq
\frac{2\xi\tan\left(\frac{\pi}{2\xi}\right)}{\pi\left(1+x^2\tan\left(\frac{\pi}{2\xi}\right)^2\right)}\le\frac{2\xi\tan\left(\frac{\pi}{2\xi}\right)}{\pi\left(1+\tan\left(\frac{\pi}{2\xi}\right)^2\right)} = \frac{\sin\left(\frac{\pi}{2\xi}\right)\cos\left(\frac{\pi}{2\xi}\right)}{\frac{\pi}{2\xi}}<1,
\eeq
we have that \eqref{eq:diff_l} is positive, and part \ref{Lphia} of the lemma is proved.

Now consider \eqref{eq:diff_L} for $\phi = \al$. To prove part \ref{Lphib} of the lemma, we need to show that $L(\al;\alpha,\xi)- L_\be(\al; \be)<0$, or equivalently 
\beq
g(e^{i\al}) - g_\be(e^{i\al})<0,
\eeq
where $g(z)$ is defined in \eqref{eq20}, and
\beq
g_\al(z):=\lim_{\xi\to\infty} g(z).
\eeq
We will use the explicit formula \eqref{eq37} for $g'(z)$. Taking the limit of that formula as $\xi\to\infty$ and replacing $\al$ with $\be$, we find the explicit formula for $g_\be'(z)$:
\beq
g_\be'(z)= \frac{1}{2z} + \frac{1}{2z}\frac{z+1}{\sqrt{R(z)}}.
\eeq
Since $g(z)\sim \log(z)+\bigO(1/z)$ and $g_\be(z)\sim \log(z)+\bigO(1/z)$ as $z\to\infty$, we find
\beq
\begin{aligned}
g(e^{i\al}) - g_\be(e^{i\al}) &= \int_\infty^{e^{i\al}} \left(g'(z)- g_\be '(z)\right)\,dz \\
&=  \int_\infty^{e^{i\al}}\left(\frac{\xi}{\pi z}\arctan\left(\frac{(z+1)\tan\left(\frac{\pi}{2\xi}\right)}{\sqrt{R(z)}}\right) -  \frac{1}{2z}\frac{z+1}{\sqrt{R(z)}}\right)\,dz.
\end{aligned}
\eeq
We take the contour of integration to be the union of two pieces: the negative real axis from $-\infty$ to $-1$, and the arc of the unit circle which connects $-1$ to $e^{i\al}$, oriented clockwise. This gives
\beq\label{eq:diff:g}
\begin{aligned}
g(e^{i\al}) - g_\be(e^{i\al}) &= \int_{-\infty}^{-1}
\frac{\xi}{\pi z}\left(\arctan\left(\frac{(z+1)\tan\left(\frac{\pi}{2\xi}\right)}{\sqrt{R(z)}}\right) -  \frac{\pi}{2\xi}\frac{(z+1)}{\sqrt{R(z)}}\right)\,dz \\
&\quad  -\frac{\xi}{\pi}  \int_\al^\pi \left[\arctanh\left(\frac{\sqrt{2}\cos(\theta/2)\tan\left(\frac{\pi}{2\xi}\right)}{\sqrt{\cos\beta-\cos\theta}}\right) - \frac{\pi}{2\xi}\frac{\sqrt{2}\cos(\theta/2)}{\sqrt{\cos\beta-\cos\theta}}\right]\,d\theta.
\end{aligned}
\eeq
Consider first the integral over the negative real axis. Recall that $\sqrt{R(z)}<0$ for $z<1$, and note that 
\beq
0<\frac{(z+1)}{\sqrt{R(z)}}<1\,, \quad \textrm{for} \  -\infty < z< -1,
\eeq
which can be checked by noting that the ratio approaches 1 as $z\to-\infty$ and is strictly decreasing and positive for $-\infty<z<-1$. Then the general inequality
\beq
\arctan(ab) > a \arctan(b) \quad \textrm{for} \ b>0, \ 0<a<1,
\eeq
implies that the first integrand in \eqref{eq:diff:g} is strictly negative. A similar argument involving $\arctanh$ instead of $\arctan$ implies that the second integrand in \eqref{eq:diff:g} is strictly positive. Thus we have that the difference of the two integrals in \eqref{eq:diff:g} is negative, which shows that \eqref{eq:diff_L} is negative for $\phi = \al$, completing the proof of part \ref{Lphib} of the lemma.

%
%
%

\end{proof}

\begin{lem}\label{Lpiovertwo}
$$L\left(\frac\pi2;\frac\pi 2,2\right)=0=L(\phi;\pi,\xi)$$
\end{lem}

\begin{proof}
If $\xi=2$ then $\tan\frac\pi{2\xi}=\tan\frac\pi 4=1$ and
$\cos\beta=1$ so $\rho(\theta)=\frac\xi{\pi^2}$. Thus
$$L\left(\frac\pi2,\frac\pi 2,2\right)=\frac1{\pi^2}
\int^{\pi/2}_{-\pi/2}\ln( 2\left(1-\cos\left(\frac\pi 2-\theta\right)
\right))d\theta.$$
But a residue calculation shows that 
$\int^\pi_0 \ln (1-\cos \theta)d\theta=-\pi\ln 2$ so we obtain
the first equality. If $\alpha=\pi$ then $\beta=\pi$ and $\rho\equiv\frac{1}{2\pi}$ so
$$
L(\alpha,\pi,\xi)=\frac{1}{2\pi}\int_{-\pi}^{\pi}\ln(2(1-\cos(\phi-\theta)))d\theta.
$$
The result now follows from the periodicity of cosine and the first part of the lemma.
\end{proof}

\begin{lem}\label{cl}
The function $L\left(\alpha,\alpha,\frac{\pi}{\alpha}\right)$ is concave down on $[\frac{\pi}{2},\pi]$ and zero at $\frac{\pi}{2}$ and $\pi$. It is not a constant and therefore it has a unique maximum. 
\end{lem}

\begin{proof}
As indicated above, $L\left(\alpha,\alpha,\frac{\pi}{\alpha}\right)$ is equal to zero for $\alpha=\frac{\pi}{2}$ and ${\pi}$. If we set $\phi=\alpha$ and $\xi=\frac{\pi}{\alpha}$ we find

\begin{align}\label{laaa}
L\left(\alpha,\alpha,\frac{\pi}{\alpha}\right)&=\frac{1}{4\alpha}\int_{-\alpha}^{\alpha}\ln(2(1-\cos(\alpha-\theta))d\theta\nonumber\\&=\frac{1}{4\alpha}\int_{0}^{2\alpha}\ln(2(1-\cos(\theta))d\theta=\frac{1}{2}\int_{0}^{1}\ln(2(1-\cos(2\alpha\theta))d\theta.
\end{align}The second derivative is
$$
\frac{d^2}{d\alpha^2} L\left(\alpha,\alpha,\frac{\pi}{\alpha}\right)=-2\int_{0}^{1}\frac{\theta^2}{1-\cos2\alpha\theta}\,d\theta.
$$
Since the integrand is continuous the interchange of differentiation and integration is allowed and since the integral is negative the lemma follows.
\end{proof}

\begin{rem}
We note that $L\left(\alpha,\alpha,\frac{\pi}{\alpha}\right)=-\frac{{\rm Cl}_2(2\alpha)}{2\alpha}$ where 
$$
{\rm Cl}_2(\theta)=-\int_{0}^{\theta}\ln\left(2\sin\frac{t}{2}\right) dt,
$$
is Clausen's integral. It can be written as
$$
\frac{{\rm Cl}_2(\theta)}{\theta}=1-\ln|\theta|+\sum_{n=1}^{\infty}\frac{\zeta(2n)}{n(2n+1)}\left(\frac{\theta}{2\pi}\right)^n,
$$
where $\zeta$  is the Riemann $\zeta$-function.
\end{rem}

We can now prove Proposition \ref{thm:L_properties}.
\begin{proof}[Proof of Proposition \ref{thm:L_properties}]
Proposition  \ref{thm:L_properties}\ref{thm:L_properties1} follows immediately from the Euler--Lagrange conditions \eqref{eq15} and Lemma \ref{Lphiin}, Proposition   \ref{thm:L_properties}\ref{thm:L_properties2} from Lemma \ref{Lphi}\ref{Lphib} with $\tilde{L}(x)= L(2\pi x/b)$. Proposition  \ref{thm:L_properties}\ref{thm:L_properties3} follows from Lemma \ref{Lphi}\ref{Lphia} along with equation \eqref{lm10}.

According to Lemma \ref{Lphiin}, as a function of $\phi$, $L(\phi; \al, \xi)$ has its maximum at $\phi=\al$. According to Lemma \ref{Lphi} as a function of $\xi$ it has its maximum at $\xi=\pi/\al$. Lemmas \ref{Lpiovertwo} and \ref{cl} imply that $L(\al; \al, \pi/\al)$ is negative for all $\al<\pi/2$, therefore whenever $\al<\pi/2$ we have $L(\phi; \al, \xi)<0$. Since $\al=\pi/b$, this is equivalent to $\tilde L(x; b, \tilde\xi)<$ for all $b>2$, proving Proposition \ref{thm:L_properties}\ref{thm:L_properties4}.

To prove Proposition \ref{thm:L_properties}\ref{thm:L_properties5}, we simply note that as $\xi\to\infty$, the saturated region vanishes as the point $\beta$ which separates the saturated region from the band approaches $\alpha$. Thus any fixed $\phi\in (-\al,\al)$, is in the band for large enough $\xi$, and in the band we have $L(\phi) = l/2$, where $l$ is the Lagrange multiplier given in \eqref{eq39}. Since it was shown in Proposition \ref{prop:lm} that the Lagrange multiplier is negative for all $\al\in(0,\pi)$ and $\xi>b$, we find that $L(\phi; \al,\xi)<0$ for large enough $\xi$, and equivalently $\tilde L(x; b,\tilde\xi)<0$ for large enough $\tilde\xi$. Then Proposition \ref{thm:L_properties}\ref{thm:L_properties5} follows since $L(\al; \al,\xi)$ is decreasing in $\xi$, thus $\tilde L(1/2; b,\tilde\xi)$ is decreasing in $\tilde\xi$.

\end{proof}

{\bf Determination of $\tilde\xi_b$.} For $\al<\pi/2$, $L(\phi; \al, \xi)$ is negative for all $\phi$ provided that $L(\al; \al, \xi)<0$. Thus, for fixed $\al\in (0,\pi/2)$, $L(\phi; \al, \xi)$ is strictly negative provided $\xi>\xi_\al$, where $\xi_\al$ is determined implicitly by the equation $L(\al; \al, \xi_\al) =0$, or equivalently,
\begin{equation}\label{xi_alpha_def}
\int_{-\al}^\al \log\left(2(1-\cos(\al - \theta)\right)\rho(\theta)\,d\theta = 0,
\end{equation}
where the dependence on $\xi$ comes entirely from the equilibrium measure density $\rho$. Then the number $\tilde\xi_b$ is given as $\tilde\xi_b =  \xi_{\pi/b}/b$. As noted in Remark \ref{rem1}, it appears there is a very simple formula for the value of $\xi$ which solves \eqref{xi_alpha_def}. Numerical calculations provide strong evidence that
\beq 
\xi_\al = \frac{\pi}{\pi-\al} \quad \textrm{for} \quad \frac{\pi}{2}\le\al<\pi,
\eeq
but we are unfortunately unable to prove this fact analytically.

\medskip

Figure \ref{lphi1} shows the graph of $L(\phi)$ with $\al = 5\pi/6$ and $\xi<\xi_\al$.  Figure \ref{lphi2} shows the graph of $L(\phi)$ with the same value of $\al$, but now with $\xi>\xi_\al$. Note that the constant value attained by $L(\phi)$ for $-\be \le \phi \le \be$ has increased from approximately $-0.09$ in Figure \ref{lphi1} to approximately $-0.035$ in Figure \ref{lphi2}, in accordance with Proposition \ref{thm:L_properties}\ref{thm:L_properties3}. But the value at the endpoint $\phi=\al$ has decreased from a positive value in Figure \ref{lphi1} to a negative one in Figure \ref{lphi2}, in accordance with Proposition \ref{thm:L_properties}\ref{thm:L_properties2}.


\begin{figure}[ht]
    \begin{minipage}[t]{0.45\linewidth}
  \centering
  \includegraphics[width=0.7\linewidth]{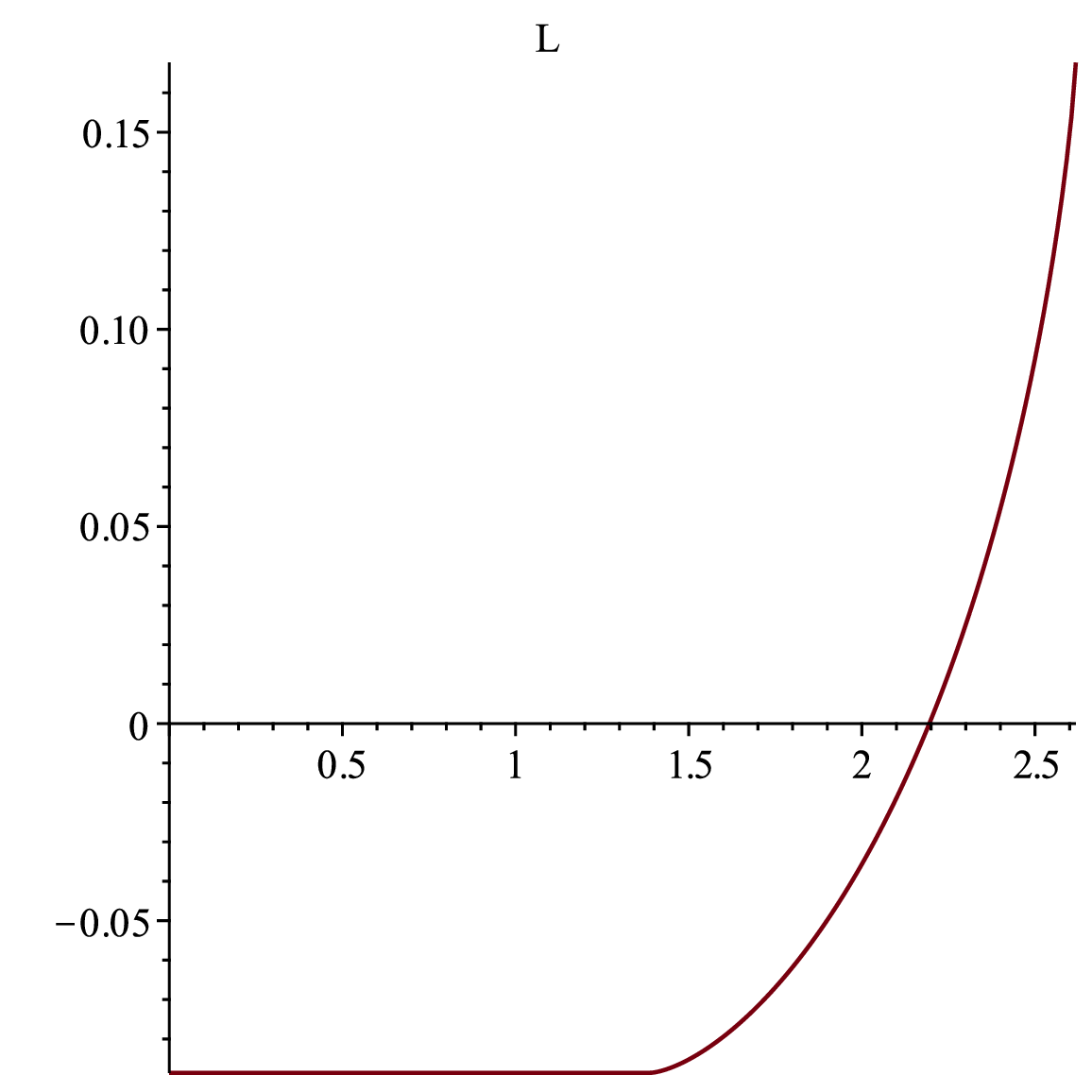}
  \caption{A plot of the function $L(\phi; \al = 5\pi/6, \xi = 32/25)$ for $0\le \phi \le\al$. For these values of $\al$ and $\xi$, the value of $\beta$ is approximately $\be \approx 1.389$. On the interval $(-\be, \be)$ the function $L(\phi)$ is constant and negative, then increases on the interval $(\be, \al)$. In this case $\xi<\xi_\al\approx 5$, so $L(\phi)$ is positive on an interval close to the endpoints $\pm \al$.}  
  \label{lphi1}
      \end{minipage}
      \hspace{\stretch{1}}
    \begin{minipage}[t]{0.5\linewidth}
      \centering
  \includegraphics[width=0.7\linewidth]{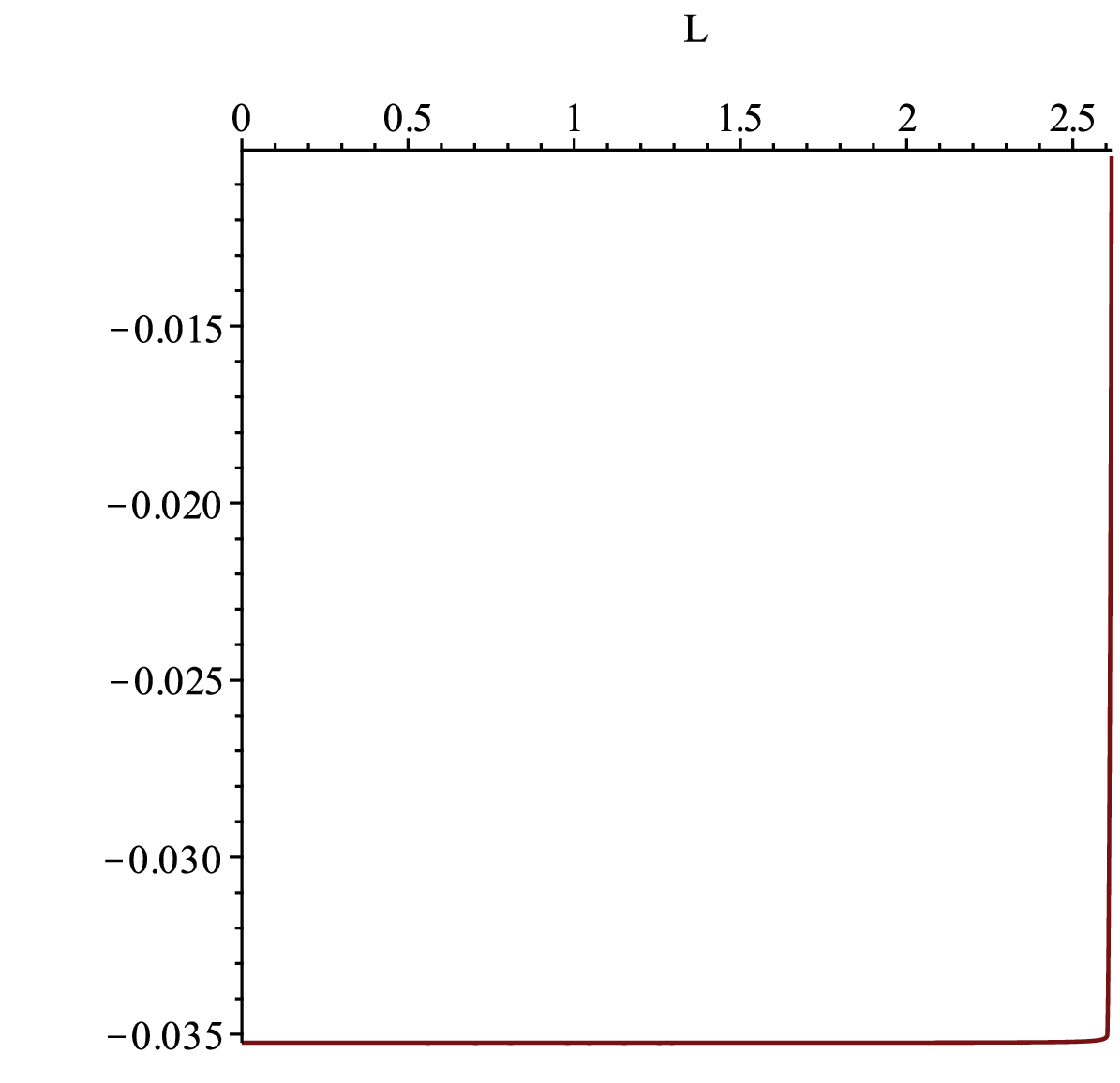}
  \caption{A plot of the function $L(\phi; \al = 5\pi/6, \xi = 7)$ for $0\le \phi \le\al$. For these values of $\al$ and $\xi$, the value of $\beta$ is approximately $\be \approx 2.604<\al\approx 2.618$. On the interval $(-\be, \be)$ the function $L(\phi)$ is constant and negative, then increases on the interval $(\be, \al)$. Noting that $\be$ is very close to $\al$, we see the function increases very rapidly over the very small interval ($\be$, $\al$). In this case $\xi>\xi_\al\approx 5$, so $L(\phi)$ remains negative on the entire interval $(-\al,\al)$.}
  \label{lphi2}
    \end{minipage}
\end{figure}

\section{Proof of Proposition \ref{subexpgr}}\label{main_tech_proof}
Before proving Proposition \ref{subexpgr} we introduce the quantities
\begin{equation}\label{in4a}
  r_{M,k} := \frac{1}{m} \sum_{z\in L_{\al,m}} z^k p_M(z), \qquad r_{M,k}^* := \frac{1}{m} \sum_{z\in L_{\al,m}}z^k p_M^*(z),
  \end{equation}
  which are related by 
  \begin{equation}\label{in4b}
r_{M,k}^*= \frac{r_{M,k+1} - r_{M+1,k}}{\rho_{M+1}}, \qquad r_{M,k+1}= \frac{r_{M,k}^* - r_{M+1,k}^*}{\rho_{M+1}}, \qquad  r_{M,0}^* = h_M.
\end{equation}

We note that since $z^m=1$ for any $z\in L_{\al,m}$ the above quantities are periodic in $k$ with period $m$. Also as noted earlier there are no orthogonal polynomials with degree greater than $m\alpha/2\pi$. Therefore we shall always assume that $m, M$ and $k$ satisfy the constraints given in the hypotheses of Proposition~\ref{subexpgr} .

A rough bound on the quantities $r_{M,k}$ is obtained from the Cauchy--Schwarz inequality:
  \begin{equation}\label{in4c}
 |r_{M,k}|^2 = \left|\frac{1}{m} \sum_{x\in L_{\al,m}} x^k p_M(x) \right|^2  \le  \left(\frac{1}{m}\sum_{x\in L_{\al,m}} |x|^{2k}\right)\left(\frac{1}{m} \sum_{x\in L_{\al,m}} |p_M(x)|^2\right)=\frac{h_M}{b},
\end{equation}
so we have
  \begin{equation}\label{in4d}
 |r_{M,k}|\le \sqrt{h_M}, \qquad  |r_{M,k}^*|\le \sqrt{h_M}.
 \end{equation}  
A more precise estimate will now be given which leads to the lower bound in Theorem~\ref{lower} 
using  $K_{n,M}$ given by equation~\eqref{chrdaym}.
Using the residue theorem and assuming that $w\in C_\al$ we find
\begin{equation}\label{pa1}
\begin{aligned}
K_{n,M}(w)=&\frac{1}{2\pi i}\oint\frac{z^{n}}{mh_M(z - w)}   \left( p^*_{M}(w)\sum_{x\in \lattice} \frac{p_M(x)}{(z-x)x^{M-1}} - z p_M(w)\sum_{x\in \lattice} \frac{p_M^*(x)}{(z-x)x^{M}} \right)\,dz  \\
& \qquad - \frac{w^{n}}{mh_M}   \left( p^*_{M}(w)\sum_{x\in \lattice} \frac{p_M(x)}{(w-x)x^{M-1}} - w p_M(w)\sum_{x\in \lattice} \frac{p_M^*(x)}{(w-x)x^{M}} \right), \\ 
\end{aligned}
\end{equation}
where the integration is over a contour which encloses $\lattice$, oriented counterclockwise, and the second term is to account for the residue at $z=w$.
The second term on the right hand side of the above equation can be recast as
\begin{align*}
 &-\frac{w^{n}}{mh_M}   \left(p^*_{M}(w)\sum_{x\in \lattice} \frac{p_M(x)}{(w-x)x^{M-1}} - w p_M(w)\sum_{x\in \lattice} \frac{p_M^*(x)}{(w-x)x^{M}}\right)\\&=\frac{w^n}{m}\sum_{x\in \lattice}\frac{1}{h_M} \frac{p_M^*(w)\overline{p^*_M(x)} - w\bar x p_M(w)\overline{p_M(x)}}{(1-\bar x w)}\\&=\frac{w^n}{m}\sum_{x\in \lattice}  \sum_{k=0}^{M} \frac{1}{h_{k}} \overline{p_k(x)} p_k(w)=w^n.
\end{align*}
\begin{rem} There are many ways one could write $K_{n,M}(w)$ as a contour integral. We choose the form of the integrand in \eqref{pa1} in part because the sums in the integrand are the discrete Cauchy transforms of $p_M$ and $p_M^*$ which appear in the second column of the Riemann--Hilbert problem solution \eqref{IP4}. Asymptotic formulas for these quantities as $M\to\infty$ follow from the steepest descent analysis presented in Section \ref{RH_analysis}. In principle then, the integral in \eqref{pa1} could be evaluated asymptotically by replacing the Cauchy transforms in the integrand with their asymptotic expressions, and performing classical steepest descent analysis. This approach should give asymptotic formulas for $K_{n,M}(w)$ in the regime $n = aM$ for $a>1$. In the current paper we do not pursue this approach, and instead derive estimates on $K_{n,M}(w)$ for $n = M+k$ for fixed $k\in \N$ as $M\to\infty$.
\end{rem}

Let us evaluate the integral in \eqref{pa1} by computing its residue at infinity. Denote the integrand of that integral by $H(z)$ so that the first term in \eqref{pa1} is
\begin{equation}
\frac{1}{2\pi i}\oint H(z)\,dz.
\end{equation}
 Expanding $H(z)$ at $z=\infty$ gives
\begin{equation}\label{pa2}
\begin{aligned}
&H(z)=\frac{z^{n-2}}{ mh_M}\left(\sum_{k=0}^\infty \left(\frac{w}{z}\right)^k\right) \left( p^*_{M}(w)\sum_{k=0}^\infty \sum_{x\in \lattice} \left(\left(\frac{x}{z}\right)^k\frac{p_M(x)}{x^{M-1}}\right) - z p_M(w)\sum_{k=0}^\infty \sum_{x\in \lattice} \left(\left(\frac{x}{z}\right)^k\frac{p_M^*(x)}{x^{M}}\right) \right) \\
&=\frac{z^{n-2}}{ mh_M}\left(\sum_{k=0}^\infty \left(\frac{w}{z}\right)^k\right) \left( p^*_{M}(w)\sum_{k=0}^\infty \sum_{x\in \lattice} \frac{x^{k+1} p_M(x)}{z^{M+k}} - z p_M(w)\sum_{k=0}^\infty \sum_{x\in \lattice} \frac{x^kp_M^*(x)}{z^{M+k-1}}\right) \\
&=z^{n-2}\left(\sum_{k=0}^\infty \left(\frac{w}{z}\right)^k\right) \left(-\frac{p_M(w)}{z^{M-1}}+\sum_{k=1}^\infty \frac{p_M^*(w)r_{M,k}-p_M(w)r_{M,k}^*}{z^{M+k-1}h_M}\right). \\
\end{aligned}
\end{equation}
The residue of $H(z)$ at $z=\infty$ is the residue of $-z^{-2} H(z^{-1})$ at $z=0$. Making this change of variables we find as $z\to 0$,
\begin{equation}
-\frac{1}{z^2} H(z^{-1}) = z^{M-n-1} \left(\sum_{k=0}^\infty \left(wz\right)^k\right)\left(p_M(w)-\sum_{k=1}^\infty \left[\frac{p_M^*(w)r_{M,k}-p_M(w)r_{M,k}^*}{h_M}\right]z^{k}\right).
\end{equation}
If $n=M$, then the residue of this function at $z=0$ is exactly $p_M(w)$. Generally for $n>M$ we have
\begin{align}\label{wknm}
K_{n,M}(w)&=w^n+\underset{z=0}\Res\left(-\frac{1}{z^2} H(z^{-1})\right)\nonumber\\&= w^n+ w^{n-M}p_M(w)-\sum_{k=1}^{n-M} \left[\frac{p_M^*(w)r_{M,k}-p_M(w)r_{M,k}^*}{h_M}\right]w^{n-M-k}.
\end{align}
 
From the definition of $r^*_{M,k}$ we see that
\begin{equation}\label{rstar}
r^*_{M,k}=\frac{1}{m}\sum_{x\in L_{\alpha,m}} x^{k+M}\overline{p_M(x)}.
\end{equation}
The expansion
$$
x^l=\sum_{j=0}^{l} x_{l,j} p_j(x),
$$
leads to
\begin{equation}\label{rstarn}
r^*_{M,k}=x_{M+k,M} h_M,
\end{equation}
where $x_{M,M}=1$ since $p_M(x)$ is monic. Write
$$
p_M(x)=\sum_{l=0}^M p_{M,j} x^j,
$$
with $p_{M,M}=1$ and
substitute this into the equation for $x^M$. Then the 
 coefficient of $x^{i}$ is,
\begin{equation}\label{integ}
x_{M,i}=-\sum_{j=i+1}^{M} x_{M,j} p_{j,i}.
\end{equation}
The recurrence formulas (equation~\eqref{in3a}) for $p_M$ and $p^*_M$ and the reality of the coefficients in $p_M$ gives
\begin{equation}\label{pnzone}
p_{M+1,j}=p_{M,j-1}-\rho_{M+1} p_{M,M-j}
\end{equation}
and
\begin{equation}\label{pnztwo}
p_{M+1,M+1-j}=p_{M,M-j}-\rho_{M+1} p_{M,j-1},
\end{equation}
where $p_{M,0}=-\rho_M$.
For $j=M$ equation~\eqref{pnzone} gives
$$
p_{M+1,M}=p_{M,M-1}+\rho_M\rho_{M+1},
$$
so
\begin{equation}\label{pnn1}
p_{M+1,M}=-\rho_1+\sum_{i=1}^M\rho_{i+1}\rho_{i}.
\end{equation}
For $i=M-1$, $M-2$ in equation~\eqref{integ} and we find
$$
x_{M,M-1}= -p_{M,M-1},
$$
$$
x_{M,M-2}=-p_{M,M-2}-x_{M,M-1}p_{M-1,M-2}.
$$
Thus 
\begin{equation}\label{xn1n}
x_{M+1,M}=\rho_1-\sum_{i=1}^M\rho_{i+1}\rho_{i}.
\end{equation}
With $M\to M+1$ and $j= M$ in equation~\eqref{pnzone} and $M\to M-1$ with $j=M-1$ in \eqref{pnztwo} 
we have
$$
p_{M+2,M}=p_{M+1,M-1}-\rho_{M+2} p_{M+1,1}=-\rho_2-\sum_{i=2}^{M+1}\rho_{i+1} p_{i,1},
$$
and
$$
p_{M,1}=p_{M-1,0}-\rho_{M} p_{M-1,M-2}=-\rho_{M-1}+\rho_{M}\rho_1-\sum_{j=1}^{M-2}\rho_{j+1}\rho_{j}.
$$
Consequently
\begin{align*}
p_{M+2,M}=&-\rho_2+\sum_{i=2}^{M+1} \rho_{i+1}\rho_{i-1}-\rho_1\sum_{i=2}^{M+1}\rho_i\rho_{i+1}+\sum_{i=3}^{M+1} \rho_i\rho_{i+1}\sum_{j=1}^{i-2}\rho_j\rho_{j+1},
\end{align*}
and
\begin{align*}
x_{M+2,M}&=-\left(-\rho_{2}+\sum_{i=2}^{M+1} \rho_{i-1}\rho_{i+1}-\rho_{1}\sum_{i=2}^{M+1}\rho_i\rho_{i+1}+\sum_{i=3}^{M+1} \rho_i\rho_{i+1}\sum_{j=1}^{i-2}\rho_j\rho_{j+1}\right)\\&-\left(\rho_1-\sum_{i=1}^{M+1}\rho_i\rho_{i+1})(-\rho_1+\sum_{i=1}^{M} \rho_i\rho_{i+1}\right)\\&=\rho_1^2+\rho_2(1-\rho_1^2)-\sum_{i=2}^{M+1} \rho_{i-1}\rho_{i+1}-\rho_1(1-\rho_2)\sum_{i=1}^M\rho_i\rho_{i+1}+\sum_{i=2}^{M+1} \rho_i\rho_{i+1}\sum_{j=i-1}^{M}\rho_j\rho_{j+1}.
\end{align*}
We now prove the following lemma.
\begin{lem}\label{lemc}
There exists and $i_0$ and a $c>0$ so that 
$$c\le -\rho_i\rho_{i+1}$$
for all $i\ge i_0$. Thus
for $k_1\ge i_0$ and $k_1\ge x_0$
$$
c\frac{(k_2-x_0)^{k+1}-(k_1-x_0)^{k+1}}{k+1}\le-\sum_{k_1}^{k_2}\rho_i\rho_{i+1}(i-x_0)^k .
$$
Also for $k_1\ge x_0$,
$$
\sum_{k_1}^{k_2}|\rho_i\rho_{i+1}|(i-x_0)^k\le\frac{(k_2+1-x_0)^{k+1}-(k_1-x_0)^{k+1}}{k+1} .
$$
\end{lem}\label{lem1}
\begin{proof}
The first inequality follows from equation\eqref{mr8}, since $0<\beta<\pi$, the poof of which  is independent of the above calculations. The inequalities for the sums are a consequence of the integral test applied to positive increasing functions. The second inequality also uses the fact that $|\rho_i|<1$ for all $i>0$.
\end{proof}
This allows

\begin{lem}\label{lemp} For $1\le k\le M$  
\begin{equation}\label{pnk}
p_{M+k,M}=p^1_{M+k,M}+p^2_{M+k,M},
\end{equation}
where
\begin{equation}\label{pnk2}
p^2_{M+k,M}=\sum_{i_1=2k-1}^{M+k-1}\rho_{i_1}\rho_{i_1+1}\sum_{i_2=2k-3}^{ i_1-2}\rho_{i_2}\rho_{i_2+1}\cdots\sum_{i_k=1}^{i_{k-1}-2}\rho_{i_k}\rho_{i_k+1},
\end{equation}
and for fixed $k$ there is a positive constant $e_k$ independent of $M$ such that              
\begin{equation}\label{pnk1}
|p^1_{M+k,M}|\le e_k M^{k-1}.
\end{equation}
Furthermore for $M\ge k$
\begin{equation}\label{uppp2}
|p^2_{M+k,M}|\le \frac{M^{k}}{k!}.
\end{equation}
\end{lem}
\begin{proof} The above examples show that the result is true for $k=1$ and $k=2$. From equation~\eqref{pnzone} we find
\begin{equation}\label{pnkn}
p_{M+k,M}=-\rho_{k}-\sum_{i=k}^{M+k-1}\rho_{i+1} p_{i,k-1},
\end{equation}
where the fact that $p_{k,0}=-\rho_k$ has been used to obtain the
last equation. Equation~\eqref{pnztwo} now gives 
\begin{equation}\label{pnknn}
p_{M+k,M}=p^1_{M+k,M}+p^2_{M+k,M}=-\rho_k-\sum_{i=k}^{M+k-1}\rho_{i+1}
p_{i-1,k-2}+\sum_{i=k}^{M+k-1}\rho_i\rho_{i+1} p_{i-1,i-k},
\end{equation}
so with $p_{i-1,i-k}=p^1_{i-1,i-k}+p^2_{i-1,i-k}$ we find that
$$
p^2_{M+k,M}=\sum_{i=k}^{M+k-1}\rho_i\rho_{i+1} p^2_{i-1,i-k}=\sum_{i=0}^{M-1}\rho_{i+k}\rho_{i+k+1} p^2_{i+k-1,i}
$$
and equation~\eqref{pnk2} now follows from the induction hypothesis. 
Since $\sum_{i_k=1}^{i_{k-1}-2}=i_{k-1}-2$, the upper bound on $|p^2_{M+k,M}|$ follows from the last assertion of Lemma~\ref{lemc} and the fact that $|\rho_i|<1$. 
To obtain the upper bound on $p^1_{m+k,M}$ we find from equation~\eqref{pnknn}
\begin{equation}\label{pnkk}
p^1_{M+k,M}=-\rho_k-\sum_{i=k}^{M+k-1}\rho_{i+1}
p_{i-1,k-2}+\sum_{i=0}^{M-1}\rho_{i+k}\rho_{i+k+1} p^1_{i+k-1,i}.
\end{equation}
By induction
\begin{equation}\label{ind1}
\left|\sum_{i=0}^{M-1}\rho_{i+k}\rho_{i+k+1}
p^1_{i+k-1,i}\right|<\sum_{i=0}^{M-1}e_{k-2}i^{k-2}\le e_{k-2}\frac{M^{k-1}}{k-1}
\end{equation}
The use of equation~\eqref{pnzone} in the second sum yields  
\begin{align*}
\left|\sum_{i=k}^{M+k-1}\rho_{i+1}
 p_{i-1,k-2}\right|&\le\left|\sum_{i=k}^{M+k-1}\rho_{i+1} p_{i-2,k-3}\right|+\left|\sum_{i=k}^{M+k-1}\rho_{i+1}\rho_{i-1}p_{i+k-2,i}\right|\\&\le\left|\sum_{i=k}^{M+k-1}\rho_{i+1} p_{i-2,k-3}\right|+\sum_{i=k}^{M+k-1}e_{k-2}i^{k-2}\\&\le\left|\sum_{i=k}^{M+k-1}\rho_{i+1} p_{i-2,k-3}\right|+e_{k-2}\frac{M^{k-1}}{k-1},
\end{align*}
where equation~\eqref{ind1} and the bound proved for $p^2_{i+k-2,i}$ have been used to obtain the second inequality. Repeated use  $k-2$ times of  equation~\eqref{pnzone} in the sum on the last line of the  above
equation shows that it is $\bigO(M^{k-2})$. 
Thus
$$
\left|p^1_{M+k,M}\right|\le e_k M^{k-1}.
$$


\end{proof}
\begin{lem}\label{lempx} For $1\le k\le M$ 
\begin{equation}\label{xnk}
x_{M+k,M}=x^1_{M+k,M}+x^2_{M+k,M},
\end{equation}
where
\begin{equation}\label{xnk2}
x^2_{M+k,k}=(-1)^k\sum_{i_1=k}^{M+k-1}\rho_{i_1}\rho_{i_1+1}\sum_{i_2=i_1-1}^{M+k-2}\rho_{i_2}\rho_{i_2+1}\cdots\sum_{i_k=i_{k-1}-1}^{M}\rho_{i_k}\rho_{i_k+1}.
\end{equation}
For $M>i_0$
\begin{equation}\label{ux2}
\frac{(M+k-1)^{k}}{k!}\ge\left|x^2_{M+k,M}\right|\ge c^k\frac{(M-i_0+1)^k}{k!}-\frac{1}{k!}((M+k-1)^k-(M-i_o+k)^k)
\end{equation}
For fixed $k$ there is a  positive constant $t_k$ independent of $M$ such that,
\begin{equation}\label{xnk1}
\left|x^1_{M+k,M}\right|\le t_k M^{k-1}.
\end{equation}
\end{lem}
\begin{proof}
Equation~\eqref{xn1n} gives the result for $k=1$. For general $k$ equation~\eqref{integ} is
\begin{equation}\label{xnkn}
x_{M+k,M}=-p_{M+k,M}-x_{M+k,M+k-1}p_{M+k-1,M}\cdots-x_{M+k,M+1}p_{M+1,M}.
\end{equation}
With the use of Lemma~\ref{lemp} we have
\begin{align}\label{xnk11}
&p_{M+k,M}+x_{M+k,M+k-1}p_{M+k-1,M}\nonumber\\&=\left(p^1_{M+k,M}+p^2_{M+k,M}\right)+(x^1_{M+k,M+k-1}+x^2_{M+k-1,M+k-1})\left(p^1_{M+k-1,M}+p^2_{M+k-1,M}\right).
\end{align}
Extracting the part that only contains $p^2$ yields
\begin{align*}
&p^2_{M+k,M}+p^2_{M+k-1,M}x^2_{M+k,M+k-1}\\&=\sum_{i_1=k}^{M+k-1}\rho_{i_1}\rho_{i_1+1}\sum_{i_2=k-1}^{ i_1-2}\rho_{i_2}\rho_{i_2+1}\sum_{i_3=k-2}^{ i_2-2}\rho_{i_3}\rho_{i_3+1}\cdots\sum_{i_k=1}^{i_{k-1}-2}\rho_{i_k}\rho_{i_k+1}\\&-\sum_{i_1=1}^{M+k-1}\rho_{i_1}\rho_{i_1+1}\left(\sum_{i_2=k-1}^{ M+k-2}\rho_{i_2}\rho_{i_2+1}\sum_{i_3=k-2}^{ i_2-2}\rho_{i_3}\rho_{i_3+1}\cdots\sum_{i_k=1}^{i_{k-1}-2}\rho_{i_k}\rho_{i_k+1}\right).
\end{align*}
Extending the first sum in the above equation to $i_1=1$ yields
$$
p^2_{M+k,M}+x^2_{M+k-1,M+k-1}p^2_{M+k-1,M}=-\sum_{i_1=1}^{M+k-1}\rho_{i_1}\rho_{i_1+1}\sum_{i_2=i_1-1}^{
  M+k-2}\rho_{i_2}\rho_{i_2+1}\cdots\sum_{i_k=1}^{i_{k-1}-2}\rho_{i_k}\rho_{i_k+1}.
$$
Thus
\begin{align*}
&p^2_{M+k,M}+x^2_{M+k-1,M+k-1}p^2_{M+k-1,M}+x^2_{M+k-1,M+k-2}p^2_{M+k-2,M}\\&=-\sum_{i_1=1}^{M+k-1}\rho_{i_1}\rho_{i_1+1}\sum_{i_2=i_1-1}^{M+k-2}\rho_{i_2}\rho_{i_2+1}\sum_{i_3=k-2}^{ i_2-2}\rho_{i_2}\rho_{i_2+1}\cdots\sum_{i_k=1}^{i_{k-1}-2}\rho_{i_k}\rho_{i_k+1}\\&+\sum_{i_1=1}^{M+k-1}\rho_{i_1}\rho_{i_1+1}\sum_{i_2=i_1-1}^{ M+k-2}\rho_{i_2}\rho_{i_2+1}\sum_{i_3=k-2}^{M-k-3}\rho_{i_3}\rho_{i_3+1}\cdots\sum_{i_k=1}^{i_{k-1}-2}\rho_{i_k}\rho_{i_k+1}\\&=\sum_{i_1=1}^{M+k-1}\rho_{i_1}\rho_{i_1+1}\sum_{i_2=i_1-1}^{ M+k-2}\rho_{i_2}\rho_{i_2+1}\sum_{i_3=i_2-1}^{M-k-3}\rho_{i_3}\rho_{i_3+1}\cdots\sum_{i_k=1}^{i_{k-1}-2}\rho_{i_k}\rho_{i_k+1}.
\end{align*}
Continuing on gives equation~\eqref{xnk2} once the empty sums have been removed. Since $\sum_{i_k=i_{k-1}-1}^M=M-i_{k-1}+2$ the upper bound on $x^2_{M+k,k}$ follows from the integral test and the fact that $|\rho_i|<1$. To obtain the lower bound we write
$$
x^2_{M+k,k}=A_{M,k,i_0}+B_{M,k,i_0},
$$
where with the use of Lemma~\ref{lemc}
\begin{align*}
A_{M,k,i_0}=&\sum_{i_1=i_0+k-1}^{M+k-1}|\rho_{i_1}\rho_{i_1+1}|\sum_{i_2=i_1-1}^{M+k-2}|\rho_{i_2}\rho_{i_2+1}|\cdots\sum_{i_k=i_{k-1}-1}^{M}|\rho_{i_k}\rho_{i_k+1}|\\&\ge c^k\sum_{i_1=i_0+k-1}^{M+k-1}\sum_{i_2=i_1-1}^{M+k-2}\cdots\sum_{i_k=i_{k-1}-1}^{M}\ge c^k\frac{(M-i_0+1)^k}{k!},
\end{align*}
and
$$
B_{M,k,i_0}=-\sum_{i_1=k}^{i_0+k-2}\rho_{i_1}\rho_{i_1+1}\sum_{i_2=i_1-1}^{M+k-2}\rho_{i_2}\rho_{i_2+1}\cdots\sum_{i_k=i_{k-1}-1}^{M}\rho_{i_k}\rho_{i_k+1}.
$$
Since $\sum_{i_k=i_{k-1}-1}^M=M-i_{k-1}+2$ from the integral test  and the fact that $|\rho_i|<1$ we find
$$
|B_{M,k,i_0}|\le\sum_{i_1=k}^{i_0+k-2}\frac{(M-i_i+2k-2)^{k-1}}{(k-1)!}\le\frac{1}{k!}\left((M+k-1)^k-(M-i_0+k)^k\right),
$$
which gives the result.
To prove the last assertion we look at the remaining terms in
equation~\eqref{xnk11} and observe, using induction, Lemma~\ref{lemc}, and the previous upper bound on $|x^2_{M+k,k}|$, that
$$
\left|p^1_{M+k,M}+\left(x^1_{M+k,M+k}+x^2_{M+k-1,M+k-1}\right)p^1_{M+k-1,M}+x^1_{M+k,M+k}p^2_{M+k-1,M}\right|\le \tilde t_k M^{k-1}.
$$ 
Continuing this for $k-2$ steps gives the result.

\end{proof}
We now give the proof of Proposition~\eqref{subexpgr}:
\begin{proof}
With the substitution of $n=M+k$ and 
$r_{M,k}=\frac{r^*_{M,k-1}-r^*_{M+1,k-1}}{\rho_{M+1}}$ in equation~4.8 we find,
$$
\text{Res}_{z=0}(-\frac{G(1/z)}{z^2})=p^*_M(w)\frac{p_M(w)}{p^*_M(w)}\frac{r^*_{M,k}}{h_M}+w^{k}\frac{p_M(w)}{p^*_M(w)}-\sum_{l=1}^{k}\frac{r^*_{M,l-1}-r^*_{M+1,l-1}}{\rho_{M+l}h_M}w^{k-l}-\sum_{l=1}^{k-1}\frac{p_M(w)}{p^*_M(w)}\frac{r^*_{M,l}}{h_M}w^{k-l}.
$$
Equation~\eqref{subexpgr} now follows from Lemma~\ref{lempx}, equation~\eqref{rstarn}, and the fact that
$|\frac{p_M(w)}{p^*_{M}(w)}|=1$ for $w$ on the unit circle. Note we have also used the fact that from equation~\eqref{in3d} $h_{M+1}=(1-\rho_{M+1}^2)h_M$. Equation~\eqref{wknm}, Theorem~\ref{thmasym2} and Lemma~\ref{cl} finish the result.
\end{proof}
%
%


\section{Proof of propositions \ref{EqMeasure}, \ref{prop:lm}, and \ref{g-function}}\label{eq_measure}
In this section we compute the equilibrium measure and related quantities. The calculation of the equilibrium measure is based on its resolvent, 
\begin{equation}\label{lx1}
\om(z)=\int_{-\al}^\al \frac{\rho(\th)d\th}{z-e^{i\th}}\,.
\end{equation}
Notice that $\om(z)=g'(z)$, where the $g$-function $g(z)$ is defined in \eqref{eq20}. We therefore first record some properties of the $g$-function. For $-\al\le \theta\le \al$, the function $\log(z-e^{i\theta})$ the asymptotics as $z\to+\infty$,
\begin{equation}\label{eq23}
\log(z-e^{i\th})=\log z+O(z^{-1}).
\end{equation}
When taking the principal branch of the logarithm $\log(z-e^{i\theta})$ with $z=e^{\phi} \in \T$, it is easy to check that the imaginary part $\arg(e^{i\phi}-e^{i\th})$ is given by
\begin{equation}\label{eq24}
\arg(e^{i\phi}-e^{i\th})=\frac{\phi+\th}{2}+\frac{\pi}{2}\,,\quad \textrm{if}\quad \pi>\phi>\th>-\pi,
\end{equation}
and 
\begin{equation}\label{eq25}
\begin{aligned}
\arg_{\pm}(e^{i\phi}-e^{i\th})&=\frac{\phi+\th}{2}+\frac{\pi}{2}\pm\pi\,,\quad \textrm{if}\quad \pi>\th>\phi>-\pi.
\end{aligned}
\end{equation}
Therefore, for all $\phi,\th\in(-\pi,\pi)$,
\begin{equation}\label{eq26}
\arg_+(e^{i\phi}-e^{i\th})+\arg_-(e^{i\phi}-e^{i\th})=\phi+\th+\pi,
\end{equation}
and hence
\begin{equation}\label{eq27}
\log_+(e^{i\phi}-e^{i\th})+\log_-(e^{i\phi}-e^{i\th})=2\log|e^{i\phi}-e^{i\th}|+i(\phi+\th+\pi)\,.
\end{equation}
Also,
\begin{equation}\label{eq28}
\arg_+(e^{i\phi}-e^{i\th})-\arg_-(e^{i\phi}-e^{i\th})=\left\{
\begin{aligned}
&0,\quad \textrm{if}\quad \pi>\phi>\th>-\pi,\\
&2\pi ,\quad \textrm{if}\quad \pi>\th>\phi>-\pi,
\end{aligned}
\right.
\end{equation}
and hence
\begin{equation}\label{eq29}
\log_+(e^{i\phi}-e^{i\th})-\log_-(e^{i\phi}-e^{i\th})=\left\{
\begin{aligned}
&0,\quad \textrm{if}\quad \pi>\phi>\th>-\pi,\\
&2\pi i ,\quad \textrm{if}\quad \pi>\th>\phi>-\pi.
\end{aligned}
\right.
\end{equation}

The following proposition, presented without proof, collects some of the important analytical properties of the $g$-function.

\begin{prop} The $g$-function has the following properties:
\begin{enumerate}
\item $g$ is analytic on $\C\setminus\big((-\infty,-1]\cup\T \big)$.
\item On $(-\infty,-1)$, $g_+(z)-g_-(z)=2\pi i$.
\item $g(z)=\log z+\bigO(z^{-1})$ as $z\to\infty$.
\item $e^{M g(z)}$ is analytic on $\C\setminus C_\al$.
\item $e^{M g(z)}=z^{M}+O(z^{M-1})$ as $z\to\infty$.
\item $g(0)=\pi i$. 
\item If $z=e^{i\phi}\in\T$, then 
\begin{equation}\label{eq30}
g_+(z)+g_-(z)=2\di \int_{-\al}^\al \log|z-e^{i\th}|d\nu_{\eq}(\th)+i(\phi+\pi).
\end{equation}
\item If $z=e^{i\phi}\in C_{\al}$, then 
\begin{equation}\label{eq31}
G(z):=g_+(z)-g_-(z)=2\pi i\di \int_{\phi}^\al d\nu_{\eq}(\th).
\end{equation}
\end{enumerate}  
\end{prop}

Observe that (\ref{eq30}) follows from (\ref{eq27}), because
\begin{equation}\label{eq32}
\int_{-\al}^\al \th \,d\nu_{\eq}(\th)=0,
\end{equation}
and (\ref{eq31}) follows from (\ref{eq29}).
From (\ref{eq16}) and (\ref{eq30}) we obtain that for $z=e^{i\phi}$,
\begin{equation}\label{eq33}
g_+(z)+g_-(z)= l+\log z+i\pi,
  \quad \textrm{if}\quad \phi \in \supp \nu_{\eq}\cap \supp (\xi\sg-\nu_{\eq}).
\end{equation}

\subsection{Calculation of the equilibrium measure and its resolvent.}\label{EqMeasureProof}
We expect that the upper constraint on the equilibrium measure density is active near the endpoints of the interval $[-\al,\al]$.  Introduce then a number $\be$ with $0<\be<\al$, so that for $\th \in [-\al, -\be] \cup [\be, \al]$, $\rho(\th)\equiv \frac{\xi}{2\pi},$ and for $\th \in (-\be, \be),$ $0<\rho(\th)<\frac{\xi}{2\pi}$.
By differentiating equation (\ref{eq33}), we obtain that
\begin{equation}\label{lx2}
\om_+(e^{i\phi})+\om_-(e^{i\phi})= \frac{1}{e^{i\phi}}\,,
  \quad \textrm{if}\quad \phi \in(-\be, \be).
\end{equation}
 Also the Plemelj--Sokhotsky formula gives that
 \begin{equation}\label{lx3}
\om_-(e^{i\phi})-\om_+(e^{i\phi})= \frac{\xi}{e^{i\phi}}\,,
  \quad \textrm{if}\quad \phi \in [-\al, -\be] \cup [\be, \al].
\end{equation}
Now recall the function $R(z)$ introduced in \eqref{eq38}, and consider its square root $\sqrt{R(z)}$ with a cut on $C_\be$, taking the branch such that $\sqrt{R(z)} \sim z$ as $z\to\infty$. Let us take the contour $C_\be$ oriented such that its $+$-side is inside the unit circle, and its $-$-side is outside the unit circle. The function $\sqrt{R(z)}$ has the following properties:
\begin{enumerate}
\item For $z=e^{i\th} \in C_{\be}$,
 \begin{equation}\label{lx5}
 \sqrt{R(z)}_{\mp} = \pm \sqrt{2} e^{i\th /2} \sqrt{\cos \th - \cos \be}\,,
 \end{equation}
 where $\th \in (-\be, \be)$.
 \item For $\th \in (\be, \al)$
  \begin{equation}\label{lx6}
 \sqrt{R(e^{i\th})} = i \sqrt{2} e^{i\th /2} \sqrt{\cos \be - \cos \th}\,,
 \end{equation}
 and for $\th \in (-\al, -\be)$
  \begin{equation}\label{lx7}
 \sqrt{R(e^{i\th})} = -i \sqrt{2} e^{i\th /2} \sqrt{\cos \be - \cos \th}\,.
 \end{equation}
\end{enumerate}

Now introduce the function
  \begin{equation}\label{lx8}
  w(z):= \frac{\om(z)}{\sqrt{R(z)}}\,.
  \end{equation}
  It satisfies the following Riemann--Hilbert problem:
  \begin{enumerate}
\item $w(z)$ is analytic for $z\in \C \setminus C_{\al}$.
\item For $z\in C_\be$, $w(z)$ satisfies the jump property
  \begin{equation}\label{lx9}
  w_-(z)-w_+(z)=\frac{\om_-(z)}{\sqrt{R(z)}_-}-\frac{\om_+(z)}{\sqrt{R(z)}_+}=\frac{\om_-(z)+\om_+(z)}{\sqrt{R(z)}_-}=\frac{1}{z\sqrt{R(z)}_-}.
  \end{equation}
\item For $z\in C_\al \setminus C_\be$, $w(z)$ satisfies the jump property
  \begin{equation}\label{lx10}
  w_-(z)-w_+(z)=\frac{\om_-(z)-\om_+(z)}{\sqrt{R(z)}}=\frac{\xi}{z\sqrt{R(z)}}.
  \end{equation}
  \item As $z\to\infty$
  \begin{equation}\label{lx11}
  w(z)\sim \frac{1}{z^2}+\dots\,.
  \end{equation}
  \end{enumerate}  
  This RHP can be solved directly using the Plemelj--Sokhotsky formula. It yields
      \begin{equation}\label{lx12}
    w(z)=w_1(z)+w_2(z),
        \end{equation}
    where
    \begin{equation}\label{lx13}
w_1(z)=\frac{1}{2\pi i} \int_{C_\be} \frac{dw}{(z-w)w \sqrt{R(w)}_-}, \qquad w_2(z)=\frac{1}{2\pi i} \int_{C_\al \setminus C_\be} \frac{\xi dw}{(z-w)w \sqrt{R(w)}}\,.
    \end{equation}
  Making the change of variable $w=e^{i\th}$ and using the formulas \eqref{lx5}-\eqref{lx7}, we find
      \begin{equation}\label{lx14}
      \begin{aligned}
      w_1(z)&=\frac{1}{2\pi\sqrt{2}} \int_{-\be}^\be \frac{e^{-i\th/2}\,d\th}{(z-e^{i\th})\sqrt{\cos \th - \cos\be}}, \\ 
      w_2(z)&=\frac{1}{2\pi i \sqrt{2}}\left(\int_{\be}^\al - \int_{-\al}^{-\be}\right)\frac{\xi\, e^{-i\th/2}\,d\th}{(z-e^{i\th})\sqrt{\cos \be - \cos\th}}.
      \end{aligned}
      \end{equation}
      The integral for $w_1$ is rather straightforward to compute. It yields
        \begin{equation}\label{lx15}
      w_1(z)= \frac{1}{2z}\left(1+\frac{1}{\sqrt{R(z)}}\right)\,.
      \end{equation}  
      Notice that $\sqrt{R(0)}=-1$, so there is no singularity at the origin.
      
      The integral for $w_2$ is more difficult to compute, but can be computed as follows. From \eqref{lx14} we have
              \begin{equation}\label{lx15a}
              \begin{aligned}
            w_2(z)&=\frac{\xi}{2\sqrt{2}\pi i}\left(\int_{\be}^\al - \int_{-\al}^{-\be}\right)\frac{\cos(\th/2)-i\sin(\th/2)}{(z-\cos\th-i\sin\th)\sqrt{\cos \be - \cos\th}}d\th \\
            &=   \frac{\xi}{\sqrt{2}\pi}\int_{\be}^\al\frac{-(z-1)\sin(\th/2)+2\cos\th\sin(\th/2)}{(z^2-2z\cos\th+1)\sqrt{\cos \be - \cos\th}}d\th, \\
     \end{aligned}
     \end{equation}
where the last equality follows from the symmetry of the intervals $[-\al,-\be]$ and $[\be,\al]$. Now make the change of variable $x=\cos\th$. Notice then that $\sin(\th/2)=\sqrt{(1-x)/2}$. Introducing the notations $A=\cos\al$ and $B=\cos\be$, we then have
       \begin{equation}\label{lx15b}
      w_2(z)=w_{21}(z)+w_{22}(z),
      \end{equation}
      where
        \begin{equation}\label{lx15c}    
        \begin{aligned}
  w_{21}(z)&=-\frac{\xi(z-1)}{2\pi}\int_A^B \frac{dx}{\sqrt{1+x}\sqrt{B-x}(z^2-2zx+1)} \\
					 &=-\frac{\xi(z-1)}{4z\pi}\int_A^B \frac{dx}{\sqrt{1+x}\sqrt{B-x}\left(\frac{z^2+1}{2z}-x\right)}\,, \\
					\end{aligned}
     \end{equation}
					 \begin{equation}\label{lx15d}    
        \begin{aligned}
	w_{22}(z)&=\frac{\xi}{\pi}\int_A^B \frac{x\,dx}{\sqrt{1+x}\sqrt{B-x}(z^2-2zx+1)} \\
					 &=\frac{\xi}{2z\pi}\int_A^B \frac{x\,dx}{\sqrt{1+x}\sqrt{B-x}\left(\frac{z^2+1}{2z}-x\right)}\,.
     \end{aligned}
     \end{equation}
		
Let us first compute $w_{21}(z)$. Introduce the linear fractional transformation 
  \begin{equation}\label{lx16}
	v\equiv v(x):=-\frac{(z+1)^2}{R(z)}\frac{x-B}{x+1}.
		\end{equation}
Notice that
\begin{equation}\label{lx17}
v(B)=0\,, \qquad v(-1)= \infty\,, \qquad v\left(\frac{z^2+1}{2z}\right)=-1\,.
\end{equation}
The inverse transformation is
\begin{equation}\label{lx18}
x(v)=-\frac{R(z)v-B(z+1)^2}{R(z)v+(z+1)^2}\,.
\end{equation}
Making the change of variables $x\mapsto v$ in \eqref{lx15c}, we have
\begin{equation}\label{lx19}
        \begin{aligned}
  w_{21}(z)&=-\frac{\xi(z-1)}{2\pi}\int_A^B \frac{x'(v)\,dv}{\sqrt{1+x(v)}\sqrt{B-x(v)}\left(z^2-2zx(v)+1\right)} \\
	        &=-\frac{\xi(z-1)}{2\pi \sqrt{R(z)}(z+1)}\int_0^{v(A)} \frac{dv}{(1+v)\sqrt{v}}. \\     
					\end{aligned}
     \end{equation}		
		Now letting $v=y^2$, we have
\begin{equation}\label{lx20}
        \begin{aligned}
  w_{21}(z)&=-\frac{\xi(z-1)}{\pi \sqrt{R(z)}(z+1)}\int_0^{\sqrt{v(A)}} \frac{dy}{(1+y^2)} \\  
		&=-\frac{\xi(z-1)}{\pi \sqrt{R(z)}(z+1)}\arctan\left(\sqrt{v(A)}\right)\,. \\    
					\end{aligned}
     \end{equation}
		The argument of $\arctan$, $\sqrt{v(A)}$, is of course a function of $z$, so let us write it as
\begin{equation}\label{lx21}
S(z):=v(A)=\frac{(z+1)^2(B-A)}{R(z)(A+1)}\,,
\end{equation}
so that 
\begin{equation}\label{lx22}
        \begin{aligned}
  w_{21}(z)&=-\frac{\xi(z-1)}{\pi \sqrt{R(z)}(z+1)}\arctan\left(\sqrt{S(z)}\right)\,. \\    
					\end{aligned}
     \end{equation}
The arctangent is taken with the usual cut on $i\R\setminus [-i, i]$. Using \eqref{lx6} and \eqref{lx7} we see that $\sqrt{S(z)}$ maps the arc $\{e^{i\theta} : \be\le \theta \le \al\}$ to ray $(-i\infty, -i)$, and maps the arc $\{e^{i\theta} : -\al\le \theta \le-\be\}$ to ray $(i, i\infty)$, thus $\arctan(\sqrt{S(z)})$ has cuts on these arcs, with an additive jump of $\pi$. It also has a jump of sign across the arc $C_\be$ due to the cut for $\sqrt{R(z)}$.

Let us now calculate $w_{22}(z)$. Making the same change of variable $x\mapsto v$ in \eqref{lx15d} we find
 \begin{equation}\label{lx23}    
 \begin{aligned}
	w_{22}(z)&=\frac{\xi}{\pi\sqrt{R(z)}(z+1)}\int_0^{v(A)} \frac{(B(z+1)^2-R(z)v)\,dv}{\sqrt{v}(1+v)((z+1)^2+R(z)v)} \\
	&=\frac{2\xi}{\pi\sqrt{R(z)}(z+1)}\int_0^{\sqrt{v(A)}} \frac{(B(z+1)^2-R(z)y^2)\,dy}{(1+y^2)((z+1)^2+R(z)y^2)} \\
	&=\frac{\xi}{\pi z\sqrt{R(z)}(z+1)}\int_0^{\sqrt{v(A)}} \left[\frac{z^2+1}{1+y^2}-\frac{R(z)(z+1)^2}{(z+1)^2+R(z)y^2}\right]\,dy \\
	&=\frac{\xi}{\pi z\sqrt{R(z)}(z+1)}\left[(z^2+1)\arctan\left(\sqrt{S(z)}\right)-\sqrt{R(z)}(z+1)\arctan\left(\frac{\sqrt{R(z)S(z)}}{z+1}\right)\right], \\
\end{aligned}
\end{equation}
	which simplifies to
\begin{equation}\label{lx24}    
\begin{aligned}
w_{22}(z)&=\frac{\xi}{\pi z}\left[\frac{(z^2+1)\arctan\left(\sqrt{S(z)}\right)}{\sqrt{R(z)}(z+1)}-\arctan\left(\sqrt{\frac{B-A}{A+1}}\right)\right]
\end{aligned}
\end{equation}	
	Adding $w_{21}(z)$ and $w_{22}(z)$ we find
\begin{equation}\label{lx25}    
\begin{aligned}
w_2(z)=w_{21}(z)+w_{22}(z)=\frac{\xi}{\pi z}\left[\frac{\arctan(\sqrt{S(z)})}{\sqrt{R(z)}}-\arctan\left(\sqrt{\frac{B-A}{A+1}}\right)\right].
 \end{aligned}
\end{equation}			
        We can now recover the resolvent $\om(z)$:
 \begin{equation}\label{lx27}     
 \begin{aligned}
            \om(z)&=\sqrt{R(z)} [w_1(z)+w_2(z)] \\
            &=\frac{\sqrt{R(z)}}{z}\left(\frac{1}{2}-\frac{\xi\arctan\left(\sqrt{\frac{B-A}{A+1}}\right)}{\pi}\right)+\frac{1}{z}\left(\frac{1}{2}+\frac{\xi\arctan(\sqrt{S(z)})}{\pi}\right).
\end{aligned}
\end{equation}
            
 It remains to determine the value of $\be$. This can be determined by the condition $\om(z)\sim\frac{1}{z}$ as $z\to\infty$. Using \eqref{lx27} and taking $z\to\infty$, we find that
\begin{equation}\label{lx28}
\begin{aligned}
            \om(z)&\sim \frac{1}{2}-\frac{\xi}{\pi}\arctan\left(\sqrt{\frac{B-A}{A+1}}\right)  \\
            &\qquad \qquad +\frac{1}{z}\left[ -B\left(\frac{1}{2}-\frac{\xi}{\pi}\arctan\left(\sqrt{\frac{B-A}{A+1}}\right)\right)+\frac{1}{2}+\frac{\xi}{\pi}\arctan\left(\sqrt{\frac{B-A}{A+1}}\right)\right] \\
            &\qquad \qquad+\bigO(z^{-2})\,.
\end{aligned}
\end{equation}
  The constant term vanishes and the ($1/z$)-term is $1$ if and only if
\begin{equation}\label{lx29}
\begin{aligned}
   \frac{1}{2}=\frac{\xi}{\pi}\arctan\left(\sqrt{\frac{B-A}{A+1}}\right).
 \end{aligned}
 \end{equation}
	Solving this equation for $B$ gives
\begin{equation}\label{lx30}
B=A+\tan^2\left(\frac{\pi}{2\xi}\right)(1+A).
\end{equation}
Using this value for $B$, the formula \eqref{lx27} simplifies to            
\begin{equation}\label{lx30a}     
\om(z)=\frac{1}{2z} +\frac{\xi}{\pi z}\arctan\left(\sqrt{S(z)}\right)\,, \qquad \sqrt{S(z)}= \frac{(z+1)\tan\left(\frac{\pi}{2\xi}\right)}{\sqrt{R(z)}},
\end{equation}
which proves Proposition \ref{g-function}.

Finally we can recover the density for the equilibrium measure. For $z=e^{i\th}\in C_\be$ we have
\begin{equation}\label{lx31}  
\begin{aligned}
\rho(\theta)&=\frac{z}{2\pi}[\om_-(z)-\om_+(z)]\,, \\
&=\frac{\xi}{2\pi^2}\left[\arctan\left(\frac{(z+1)\tan\left(\frac{\pi}{2\xi}\right)}{\sqrt{R(z)}_-}\right)-\arctan\left(\frac{(z+1)\tan\left(\frac{\pi}{2\xi}\right)}{\sqrt{R(z)}_+}\right)\right] \\
&=\frac{\xi}{\pi^2}\arctan\left(\frac{(z+1)\tan\left(\frac{\pi}{2\xi}\right)}{\sqrt{R(z)}_-}\right).
\end{aligned}
\end{equation}
Using the formula \eqref{lx5}, we can then write
 \begin{equation}\label{lx32}  
\begin{aligned}
\rho(\theta)&=\frac{\xi}{\pi^2}\arctan\left(\frac{\sqrt{2}\tan\left(\frac{\pi}{2\xi}\right)\cos(\th/2)}{\sqrt{\cos\th-\cos\be}}\right).
\end{aligned}
\end{equation}
Since the function $\arctan(\sqrt{S(z)})$ has an additive jump of $\pi$ across the arcs $C_{\al} \setminus C_\be$, we find
 \begin{equation}\label{lx33}
\rho(\th)\equiv \frac{\xi}{2\pi},
\end{equation}
for $\th\in [-\al,-\be] \cup [\be,\al]$. This proves Proposition \ref{EqMeasure}.

\subsection{Computation of the Lagrange multiplier}\label{prop_lmProof}

We now compute the value of the Lagrange multiplier $l$ in \eqref{eq15}. Using \eqref{eq31} and \eqref{eq33} with $z=1$, we find 
\begin{equation}\label{lm1}
l=2g_-(1).
\end{equation}
To compute the value of $g_-(1)$, recall that $\om(z)=g'(z)$, where $\om(z)$ is given explicitly in \eqref{lx30a}. Since $g(z)\sim \log z$ as $z\to\infty$, it follows that
\begin{equation}\label{lm2}
\begin{aligned}
l&=2\lim_{X\to\infty} \left[\log X-\int_1^X \om(z)\,dz\right] \\
&=2\lim_{X\to\infty} \left[\frac{1}{2}\log X-\frac{\xi}{\pi}\int_1^X\frac{1}{z} \arctan\left(\sqrt{S(z)}\right)\,dz\right].
\end{aligned}
\end{equation}
Using the integral representation for the arctan function we can write \eqref{lm2} as
\begin{equation}\label{lm3}
\begin{aligned}
l&=2\lim_{X\to\infty} \left[\frac{1}{2}\log X-\frac{\xi}{\pi}\int_1^X\,\frac{dz}{z}\int_0^{\sqrt{S(z)}} \frac{dx}{1+x^2}\right].
\end{aligned}
\end{equation}
Notice that the function $\sqrt{S(z)}$ is monotonically decreasing on $(1,\infty)$, with 
\begin{equation}\label{lm4}
\sqrt{S(1)}=\frac{\sqrt{2}\tan\left(\frac{\pi}{2\xi}\right)}{\sqrt{1-B}}\,, \qquad \sqrt{S(\infty)}=\tan\left(\frac{\pi}{2\xi}\right).
\end{equation}
We can thus change the order of integration in \eqref{lm3} to obtain
\begin{equation}\label{lm5}
\begin{aligned}
l&=2\lim_{X\to\infty} \left[\frac{1}{2}\log X-\frac{\xi}{\pi}\int_1^X\,\frac{dz}{z}\int_0^{\tan\left(\frac{\pi}{2\xi}\right)} \frac{dx}{1+x^2}\right]-\frac{2\xi}{\pi}\int_{\tan\left(\frac{\pi}{2\xi}\right)}^{\frac{\sqrt{2}\tan\left(\frac{\pi}{2\xi}\right)}{\sqrt{1-B}}}\,dx\int_1^{f(x)}\frac{dz}{z(1+x^2)}\,,
\end{aligned}
\end{equation}
where in the latter integral the limit in $X$ has been taken (and the indefinite integral converges) and the function $f(x)$ is the functional inverse of $\sqrt{S(z)}$ on the interval of integration:
\begin{equation}\label{lm6}
f: \left(\tan\left(\frac{\pi}{2\xi}\right), \frac{\sqrt{2}\tan\left(\frac{\pi}{2\xi}\right)}{\sqrt{1-B}}\right]\to [1,\infty)\,, \qquad f\left(\sqrt{S(z)}\right)=z\,.
\end{equation}
Explicitly we have
\begin{equation}\label{lm7}
f(x)=\frac{Bx^2+\tan\left(\frac{\pi}{2\xi}\right)^2+x\sqrt{1+B}\sqrt{2\tan\left(\frac{\pi}{2\xi}\right)^2-(1-B)x^2}}{x^2-\tan\left(\frac{\pi}{2\xi}\right)^2}\,.
\end{equation}
Simplifying \eqref{lm5} gives
\begin{equation}\label{lm8}
\begin{aligned}
l&=2\lim_{X\to\infty} \left[\frac{1}{2}\log X-\frac{\xi}{\pi}(\log X) \arctan\left(\tan\left(\frac{\pi}{2\xi}\right)\right)\right]-\frac{2\xi}{\pi}\int_{\tan\left(\frac{\pi}{2\xi}\right)}^{\frac{\sqrt{2}\tan\left(\frac{\pi}{2\xi}\right)}{\sqrt{1-B}}}\frac{\log f(x)}{1+x^2}\,dx \\
&=-\frac{2\xi}{\pi}\int_{\tan\left(\frac{\pi}{2\xi}\right)}^{\frac{\sqrt{2}\tan\left(\frac{\pi}{2\xi}\right)}{\sqrt{1-B}}}\frac{\log f(x)\,}{1+x^2}\,dx<0\,.
\end{aligned}
\end{equation}
Making the change of variable $x\mapsto \tan\left(\frac{\pi}{2\xi}\right) x$, proves \eqref{eq39}. Since $\beta\to\alpha$ as $\xi\to\infty$, we immediately obtain \eqref{lm10} as well.


\section{Riemann--Hilbert analysis }\label{RH_analysis}

In this section we perform the Riemann--Hilbert analysis for the orthogonal polynomials \eqref{def:orthoprod3}, which is the main part of the proof of the main theorems. The main idea is that the orthogonal polynomials can be encoded into the solution to a certain $2\times2$ matrix-valued Riemann--Hilbert problem (RHP) as formulated by Fokas, Its, and Kitaev \cite{Fokas-Its-Kitaev92}, and that this RHP can then be evaluated asymptotically as the degree of the orthogonal polynomials approach infinity using the steepest descent method of Deift and Zhou \cite{Deift-Zhou92}. For a description of this analysis for a general class of continuous orthogonal polynomials on the real line, see e.g. \cite{Deift-Kriecherbauer-McLaughlin-Venakides-Zhou99, Deift99, Bleher-Liechty14}. In our case we need to deal with discrete orthogonal polynomials, and the analysis is slightly different, see e.g. \cite{Bleher-Liechty14, BKMM}.

We begin with an interpolation problem which encodes the orthogonal polynomials \eqref{def:orthoprod3}.
\subsection{Interpolation Problem}\label{sec:IP}
We seek a $2 \times 2$ matrix valued function $\mathbf P_M(z)$ satisfying the following conditions.
\begin{enumerate}
\item {\it Analyticity.}  $\mathbf P_M(z)$ is analytic for all $z \in \C \setminus L_{\al, m}$.
\item {\it Residues at poles.}  The entries $\mathbf P_M(z)_{11}$ and $\mathbf P_M(z)_{21}$ are each entire functions of $z$, whereas the entries $\mathbf P_M(z)_{12}$ and $\mathbf P_M(z)_{22}$ have simple poles at each node $L_{\al, m}$ such that
    \begin{equation}\label{IP2}
    \underset{z=x}{\Res}\mathbf P_M(z)_{j2}=\frac{1}{m\, x^{M-1}} \mathbf P_M(x)_{j1}\,, \ j=1,2.
    \end{equation}
\item {\it Asymptotics at infinity.}  As $z\to \infty$, $\mathbf P_M(z)$ admits the expansion
    \begin{equation}\label{IP3}
    \mathbf P_M(z) \sim \left(\I + \frac{\mathbf P_1}{z} + \frac{\mathbf P_2}{z^2}+ \cdots \right)\begin{pmatrix} z^M & 0 \\ 0 & z^{-M}\end{pmatrix}.
    \end{equation}
\end{enumerate}
It is not difficult to see that this Interpolation Problem has the unique solution,
\begin{equation}\label{IP4}
\mathbf P_M(z) =\begin{pmatrix} p_M(z) &  \frac{1}{m}\underset{x\in L_{\al, m}}{\sum} \frac{p_M(x)}{(z-x)x^{M-1}} \\ \frac{1}{h_{M-1}} p_{M-1}^*(z) &  \frac{1}{m\, h_{M-1}}\underset{x\in L_{\al, m}}{\sum} \frac{p_{M-1}^*(x) }{(z-x)x^{M-1}} \end{pmatrix}\,.
\end{equation}
Evaluating at $z=0$ gives
\begin{equation}\label{IP4a}
\mathbf P_M(0) =\begin{pmatrix} -\rho_{M} &  -h_M \\ \frac{1}{h_{M-1}}  &  -\rho_{M} \end{pmatrix}\,.
\end{equation}

While the interpolation problem provides the initial step in the use of the RH techniques the matrix ${\bf P}_M$ is not yet in the form amenable for convenient analysis since it has poles which must be removed. This is done below where we  reduce this interpolation problem to a Riemann--Hilbert problem. 
For some $\ep>0$, introduce the notations
\begin{equation}\label{red0}
\begin{aligned}
C^{\pm}_\al&:=\{ z\in \C : |z| = 1 \mp \ep \ \textrm{and} \ -\al < \arg z < \al\}\,, \\ \qquad I^{+}_\al &:= \{ z\in \C : 1-\ep \le |z| \le 1 \ \textrm{and} \ \arg z = \al\}\,, \\ I^{-}_\al &:= \{ z\in \C : 1 \le |z| \le 1+\ep  \ \textrm{and} \ \arg z = \al\}.
\end{aligned}
\end{equation}
 Introduce the function 
\begin{equation}\label{red1}
\Pi(z)=z^{m/2}-z^{-m/2},
\end{equation}
which is meromorphic if $m$ is even. If $m$ is odd, we may take the cut on the negative real axis.
This function has the property that
\begin{equation}\label{red2}
\Pi'(x_j)=\,\frac{(-1)^j m}{x_j}\,, \quad \textrm{for}\quad x_j=e^{\frac{2\pi i j}{m}}\in L_{\al, m}\,,
\end{equation}
and therefore
\begin{equation}\label{red3}
\underset{z=x_j}{\Res} \, \frac{z^{m/2}}{z^{M}\Pi(z)}=\underset{z=x_j}{\Res} \, \frac{1}{z^{M}z^{m/2}\Pi(z)}=\frac{1}{m x_j^{M-1}}.
\end{equation}

Introduce the upper triangular matrices
\begin{equation}\label{red4}
\mathbf D^u_{\pm}(z)=\begin{pmatrix} 1 & -\frac{z^{-M}}{\Pi(z)}z^{\pm m/2} \\ 0 & 1\end{pmatrix},
\end{equation}
and the lower triangular matrices
\begin{equation}\label{red4a}
\mathbf D^l_{\pm}(z)=\begin{pmatrix} \Pi(z)^{-1} & 0 \\ -z^{M}z^{\pm m/2} &  \Pi(z) \end{pmatrix}=\begin{pmatrix} \Pi(z)^{-1} & 0 \\ 0 & \Pi(z)\end{pmatrix} \begin{pmatrix} 1 & 0 \\ -\frac{z^{M} z^{\pm m/2}}{\Pi(z)} & 1 \end{pmatrix}\,.
\end{equation}
Introduce also the regions
\begin{equation}\label{red4ba}
\begin{aligned}
\Om_\pm^{\nabla}&=\{z : 0<\pm (1-|z|)<\ep \ \textrm{and} \ -\be<\arg z <\be\}\,, \\
\Om_\pm^\De&=\{z : 0<\pm (1-|z|)<\ep \ \textrm{and} \ \be<\arg z <\al \textrm{ or } -\al<\arg z<-\be\}\,.
\end{aligned}
\end{equation}
We now make the transformation
\begin{equation}\label{red5}
\mathbf R_M(z)=\left\{
\begin{aligned}
&\sg_3 \mathbf P_{M}(z) \mathbf D^u_{\pm}(z)\sg_3\,, \quad z\in \Om^\nabla_{\pm} \\
&\sg_3 \mathbf P_{M}(z) \mathbf D^l_{\pm}(z)\sg_3\,, \quad z\in \Om^\De_{\pm} \\
&\sg_3\mathbf P_{M}(z)\sg_3, \qquad \textrm{otherwise},
\end{aligned}\right.
\end{equation}
where $\sg_3=\begin{pmatrix} 1 & 0 \\ 0 & -1 \end{pmatrix}$ is the third Pauli matrix.
It is easy to check that the function $\mathbf R_M(z)$ has no poles. It has jumps on each of the arcs $C_{\al}$ and $C^{\pm}_\al$ as well as on the intervals $I^{\pm}_{\al}$, $I^{\pm}_{(- \al)}$  $I^{\pm}_{\be}$, $I^{\pm}_{(- \be)}$. Specifically, the matrix valued function satisfies the jump condition
\begin{equation}\label{red6}
\mathbf R_{M+}(z)=\mathbf R_{M-}(z)\mathbf J_R(z)\,, \quad z\in \Sg_R,
\end{equation}
for some jump function $\mathbf J_R(z)$, where $\Sg_R=C_{\al} \cup C^+_{\al} \cup C^-_{\al} \cup I^{+}_\al \cup I^{-}_\al \cup I^{+}_\be \cup I^{-}_\be\cup I^{+}_{(-\al)} \cup I^{-}_{(-\al)} \cup I^{+}_{(-\be)} \cup I^{-}_{(-\be)}.$ We consider the following orientation of $\Sg_R$: $C_\al$ and $C^-_\al$ are oriented counterclockwise; $C^+_\al$ is oriented clockwise; $I^\pm_\al$ and $I^{\pm}_{\be}$ are oriented towards the origin; and $I^\pm_{(-\al)}$ and $I^{\pm}_{(-\be)}$ are oriented away from the origin. Then the jump matrix $\mathbf J_R(z)$ is given by
\begin{equation}\label{red7}
\mathbf J_R(z)=\left\{
\begin{aligned}
&\sg_3 \mathbf D_-^u(z)^{-1}\mathbf D_+^u(z) \sg_3 = \begin{pmatrix} 1 & z^{-M} \\ 0 & 1 \end{pmatrix}\,, \quad \textrm{for} \ z\in C_{\be} \\
&\sg_3 \mathbf D_-^l(z)^{-1}\mathbf D_+^l(z) \sg_3 = \begin{pmatrix} 1 &0 \\  z^{M} & 1 \end{pmatrix}\,, \quad \textrm{for} \ z\in C_\al \setminus C_{\be}  \\
&\sg_3 \mathbf D_\pm^u(z)\sg_3 = \begin{pmatrix} 1 &\pm \frac{1}{z^{M}(1-z^{\mp m})} \\ 0 & 1 \end{pmatrix}\,, \quad \textrm{for} \ z\in C^{\pm}_{\be} \\
&\sg_3 \mathbf D_\pm^l(z)\sg_3 = \begin{pmatrix} \Pi(z)^{-1} & 0 \\ z^M z^{\pm m/2} & \Pi(z) \end{pmatrix}\,, \quad \textrm{for} \ z\in \{C^{\pm}_{\al}\setminus C^{\pm}_{\be}\} \cup I^{\pm}_\al \cup I^{\pm}_{(-\al)} \\
&\sg_3 \mathbf D_\pm^l(z)^{-1}\mathbf D_\pm^u(z) \sg_3=\begin{pmatrix} \Pi(z) & z^{-M} z^{\pm m/2} \\ -z^M z^{\pm m/2} & \mp z^{\pm m/2} \end{pmatrix}\,, \quad \textrm{for} \ z\in I^{\pm}_{\be}\cup I^{\pm}_{(-\be)}.
\end{aligned}\right.
\end{equation}

\subsection{First transformation of the RHP}
We make the change of variables
\begin{equation}\label{ft1}
\mathbf R_M(z)=e^{\frac{M(l+i\pi)}{2}\sg_3}\mathbf T_M(z) e^{M\left(g(z)-\frac{l+i\pi}{2}\right)\sg_3}.
\end{equation}
Then $\mathbf T_M$ satisfies the following RHP:
\begin{enumerate}
\item $\mathbf T_M(z)$ is analytic on $\C \setminus \Sg_R$.
\item $\mathbf T_{M+}(z)=\mathbf T_{M-}(z)\mathbf J_T(z)$ where
\begin{equation}\label{ft2}
\mathbf J_T(z)=\left\{
\begin{aligned}
&e^{M\left(g(z)-\frac{l+i\pi}{2}\right)\sg_3}\mathbf J_R(z)e^{-M\left(g(z)-\frac{l+i\pi}{2}\right)\sg_3}\,, \quad z\in \Sg_R\setminus C_\al \\
&e^{M\left(g_{-}(z)-\frac{l+i\pi}{2}\right)\sg_3}\mathbf J_R(z)e^{-M\left(g_{+}(z)-\frac{l+i\pi}{2}\right)\sg_3}\,,\quad z\in C_\al.
\end{aligned}\right.
\end{equation}
\item As $z\to\infty$,
\begin{equation}\label{ft3}
\mathbf T_M(z)=\mathbf I+\frac{\mathbf T_1}{z}+\frac{\mathbf T_2}{z^2}+\cdots
\end{equation}
\end{enumerate}
More specifically, the jump functions are given by
\begin{equation}\label{ft4}
\mathbf J_T(z)=\left\{
\begin{aligned}
& \begin{pmatrix} e^{-MG(z)} & e^{M(g_+(z)+g_-(z)-l-i\pi -\log(z))} \\ 0 & e^{MG(z)} \end{pmatrix}\,, \quad \textrm{for} \ z\in C_{\be} \\
&\begin{pmatrix} e^{-MG(z)} & 0 \\ e^{-M(g_+(z)+g_-(z)-l -i\pi-\log(z))} & e^{MG(z)} \end{pmatrix}\,, \quad \textrm{for} \ z\in C_\al \setminus C_{\be}  \\
& \begin{pmatrix} 1 &\pm \frac{e^{M(2g(z)-l-i\pi-\log(z))}}{1-z^{\mp m}} \\ 0 & 1 \end{pmatrix}\,, \quad \textrm{for} \ z\in C^{\pm}_{\be} \\
& \begin{pmatrix} \Pi(z)^{-1} & 0 \\ e^{-M(2g(z)-l-i\pi-\log(z))}z^{\pm m/2} & \Pi(z) \end{pmatrix}\,, \quad \textrm{for} \ z\in \{C^{\pm}_{\al}\setminus C^{\pm}_{\be}\} \cup I^{\pm}_\al \cup I^{\pm}_{(-\al)} \\
&\begin{pmatrix} \Pi(z) & e^{M(2g(z)-l-i\pi-\log(z))} z^{\pm m/2} \\ -e^{-M(2g(z)-l-i\pi-\log(z))} z^{\pm m/2} & \mp z^{\pm m/2} \end{pmatrix}\,, \quad \textrm{for} \ z\in I^{\pm}_{\be}\cup I^{\pm}_{(-\be)},
\end{aligned}\right.
\end{equation} 
where we recall the function $G(z)$ defined in \eqref{eq31}.

\subsection{The second transformation of the RHP}
The Euler--Lagrange variational conditions \eqref{eq15} together with \eqref{eq30} imply that on $C_\be$ the jump matrix $\mathbf J_T$ has the form
\begin{equation}\label{st1}
\mathbf J_T(z)=\begin{pmatrix} e^{-MG(z)} & 1 \\ 0 & e^{MG(z)} \end{pmatrix}=\begin{pmatrix} 1 & 0 \\ e^{MG(z)} & 1 \end{pmatrix}\begin{pmatrix} 0 & 1 \\ -1 & 0 \end{pmatrix}\begin{pmatrix} 1 & 0 \\ e^{-MG(z)} & 1 \end{pmatrix}. \\
\end{equation}
Notice also that for $z\in C_\al \setminus C_\be$, by \eqref{eq31} and \eqref{lx27},
\begin{equation}\label{st2}
MG(z)=\left\{
\begin{aligned}
&2\pi i M \left(\al-\frac{\log z}{i}\right)\frac{\xi}{2\pi}=im\al-m\log z\,, \quad \textrm{for} \ \be < \arg z < \al\,, \\
&2\pi i M \left[1- \frac{\xi}{2\pi}\left(\frac{\log z}{i} +\al\right)\right]\al=2\pi i M-m\log z-im\al\,, \quad \textrm{for} \ -\al < \arg z < -\be\,,
\end{aligned}\right.
\end{equation}
thus
\begin{equation}\label{st2a}
e^{MG(z)}=\left\{
\begin{aligned}
&e^{im\al}z^{-m}\,, \quad \textrm{for} \ \be < \arg z < \al \\
&e^{-im\al}z^{-m}\,, \quad \textrm{for} \ -\al < \arg z < -\be.
\end{aligned}\right.
\end{equation}
We make the transformation
\begin{equation}\label{st3}
\mathbf S_M(z)=\left\{
\begin{aligned}
&\mathbf T_M(z)\mathbf J_+(z)^{-1}\,, \quad z\in \ \Om_+^\nabla \\
&\mathbf T_M(z)\mathbf J_-(z)\,, \quad z\in \ \Om_-^\nabla \\
&-\mathbf T_M(z) z^{-(m/2)\sg_3}\,, \quad z\in \ \Om_+^\De \\
&\mathbf T_M(z) z^{(m/2)\sg_3}\,, \quad z\in \ \Om_-^\De \\
&\mathbf T_M(z)\,, \quad \textrm{otherwise},
\end{aligned}\right.
\end{equation}
where
\begin{equation}\label{st4}
\mathbf J_+(z):=\begin{pmatrix} 1 & 0 \\ e^{-MG(z)} & 1 \end{pmatrix}\,, \qquad \mathbf J_-(z):=\begin{pmatrix} 1 & 0 \\ e^{MG(z)} & 1 \end{pmatrix}\,.
\end{equation}
Then $\mathbf S_M(z)$ satisfies the jump condition
\begin{equation}\label{st5}
\mathbf S_{M+}(z)=\mathbf S_{M-}(z)\mathbf J_S(z)\,,
\end{equation}
where
\begin{equation}\label{st5a}
\mathbf J_S(z)=\left\{
\begin{aligned}
& \begin{pmatrix} 0 & 1 \\ -1 & 0 \end{pmatrix}\,, \quad \textrm{for} \ z\in C_{\be} \\
&\begin{pmatrix} e^{-im\al-i\pi} & 0 \\ -e^{-M(g_+(z)+g_-(z)-l -i\pi-\log(z))} & e^{im\al+i\pi} \end{pmatrix}\,, \quad \textrm{for} \ z\in C_\al \setminus C_{\be}\,, \ \be<\arg z<\al  \\
&\begin{pmatrix} e^{im\al+i\pi} & 0 \\ -e^{-M(g_+(z)+g_-(z)-l -i\pi-\log(z))} & e^{-im\al-i\pi} \end{pmatrix}\,, \quad \textrm{for} \ z\in C_\al \setminus C_{\be}\,, \ -\al<\arg z<-\be  \\
& \mathbf J_T(z) \mathbf J_-(z)=\begin{pmatrix} 1-\frac{e^{M(g_+(z)+g_-(z)-l-i\pi-\log(z))}}{1-z^m}& -\frac{e^{M(2g(z)-l-i\pi-\log(z))}}{1-z^m} \\ e^{MG(z)} & 1 \end{pmatrix}\,, \quad \textrm{for} \ z\in C^{-}_{\be} \\
& \mathbf J_T(z) \mathbf J_+(z)^{-1}=\begin{pmatrix} 1-\frac{e^{M(g_+(z)+g_-(z)-l-i\pi-\log(z))}}{1-z^{-m}}& \frac{e^{M(2g(z)-l-i\pi-\log(z))}}{1-z^{-m}} \\ -e^{-MG(z)} & 1 \end{pmatrix}\,, \quad \textrm{for} \ z\in C^{+}_{\be} \\
& -\mathbf J_T(z) z^{-(m/2)\sg_3}=\begin{pmatrix} (1-z^m)^{-1} & 0 \\ -e^{-M(2g(z)-l-i\pi-\log(z))} & 1-z^m \end{pmatrix}\,, \quad \textrm{for} \ z\in \{C^{+}_{\al}\setminus C^{+}_{\be}\} \cup I^{+}_\al \cup I^{+}_{-\al} \\
& \mathbf J_T(z) z^{(m/2)\sg_3}=\begin{pmatrix} (1-z^{-m})^{-1} & 0 \\ e^{-M(2g(z)-l-i\pi-\log(z))} & 1-z^{-m} \end{pmatrix}\,, \quad \textrm{for} \ z\in \{C^{-}_{\al}\setminus C^{-}_{\be}\} \cup I^{-}_\al \cup I^{-}_{-\al} \\
&-z^{(m/2)\sg_3}\mathbf J_T(z)\mathbf J_+(z)^{-1}=\begin{pmatrix} 1 & -e^{MG(z)} z^m \\ 0 & 1 \end{pmatrix}, \quad \textrm{for} \ z\in I^{+}_{\be}\cup I^{+}_{-\be} \\
&z^{-(m/2)\sg_3}\mathbf J_T(z)\mathbf J_-(z)=\begin{pmatrix} 1 & e^{-MG(z)} z^{-m} \\ 0 & 1 \end{pmatrix}, \quad \textrm{for} \ z\in I^{-}_{\be}\cup I^{-}_{-\be}.
\end{aligned}\right.
\end{equation} 
Since we have assumed that $e^{im\al}=-1$, the jump on the arcs $C_\al \setminus C_\be$ is in fact exponentially close to the identity matrix.

\subsection{The model RHP}
The model Riemann--Hilbert problem is the problem obtained by disregarding all jumps of $\mathbf S_n$ which are asymptotically small as $M, N\to \infty$.  We therefore seek a $2 \times 2$ matrix $\mathbf M(z)$ solving the following RHP.
\begin{enumerate}
\item $\mathbf M(z)$ is analytic on $\C \setminus C_\be$.
\item On the contour $C_\be$, $\mathbf M$ satisfies the jump condition
\begin{equation}\label{mp1}
\mathbf M_+(z)=\mathbf M_-(z)\begin{pmatrix} 0 & 1 \\ -1 & 0\end{pmatrix}.
\end{equation}
\item As $z\to\infty$,
\begin{equation}\label{mp2}
\mathbf M(z)=\mathbf I+\frac{\mathbf M_1}{z}+\frac{\mathbf M_2}{z^2}+\cdots
\end{equation}
\end{enumerate}
The solution to this RHP is well known. Introduce the function
\begin{equation}\label{mp3}
\ga(z):=\left(\frac{z-e^{-i\be}}{z-e^{i\be}}\right)^{1/4},
\end{equation}
with a cut on $C_\be$, taking the branch such that $\ga(\infty)=1$.
Then the solution to the model RHP is
\begin{equation}\label{mp4}
\mathbf M(z):=\begin{pmatrix} \frac{\ga(z)+\ga(z)^{-1}}{2} & \frac{\ga(z)-\ga(z)^{-1}}{-2i} \\ \frac{\ga(z)-\ga(z)^{-1}}{2i} & \frac{\ga(z)+\ga(z)^{-1}}{2}\end{pmatrix}.
\end{equation}
On the cut $C_\be$ the function $\ga(z)$ takes the limiting values
\begin{equation}\label{mp5}
\ga_{\pm}(e^{i\th})=e^{-\frac{i}{4}(\be\pm \pi)} \left[\frac{\cos\th-\cos\be}{1-\cos(\be-\th)}\right]^{1/4}\,, \qquad \th\in(-\be,\be)\,.
\end{equation}
In particular this implies that the top left entry of $\mathbf M(z)$ takes the limiting values
\begin{equation}\label{mp6}
\mathbf M_{11}(e^{i\th})_{\pm}=\frac{e^{\mp i\pi/4}}{2}\left[e^{-i\be/4}\left[\frac{\cos\th-\cos\be}{1-\cos(\be-\th)}\right]^{1/4}\pm i e^{i\be/4}\left[\frac{\cos\th-\cos\be}{1-\cos(\be-\th)}\right]^{-1/4}\right]\,.
\end{equation}
On the rest of $\T$, the function $\ga(z)$ can be written as
\begin{equation}\label{mp7}
\ga(e^{i\th})=e^{-i\be/4} \left[\frac{\cos\be-\cos\th}{1-\cos(\th-\be)}\right]^{1/4}\,, \qquad \th\in(-\pi,-\be)\cup (\be,\pi)\,.
\end{equation}

\subsection{The parametrix at the band-saturated region end points}\label{sec:airy}
Consider small disks $D(x,\ep)$, centered at $x$ for $x=e^{\pm i\be}$, and let $D=D(e^{i\be},\ep)\cup D(e^{-i\be},\ep)$.  We seek a local parametrix $\mathbf U_n(z)$ defined on $D$ satisfying the following Riemann--Hilbert problem.
\begin{enumerate}
\item $\mathbf U_M(z)$ is analytic on $D \setminus \Sg_R$.
\item For $z\in D \cap \Sg_R$, $\mathbf U_M$ satisfies the jump condition 
\begin{equation}
\mathbf U_{M+}(z)=\mathbf U_{M-}(z)\mathbf J_S(z).
\end{equation}
\item As $M\to\infty$,
\begin{equation}\label{pm2}
 \mathbf U_M(z)=\mathbf M(z)(\mathbf I+\bigO(M^{-1})), \quad  \textrm{uniformly on}  \ \partial D.
 \end{equation}
\end{enumerate}
The solution to this local Riemann--Hilbert Problem is standard, and we present it here without proof.
Let $\Ai(z)$ be the Airy function \cite{Olver74}), and define the functions
\begin{equation}\label{pm7}
y_0(z)=\Ai (z), \quad y_1(z)=\omega \Ai (\omega z), \quad y_2(z)=\omega^2 \Ai (\omega^2 z),
\end{equation}
and the matrix-valued function
\begin{equation}\label{pmi5}
\Phi(z)=\left\{
\begin{aligned}
&\begin{pmatrix}y_2(z) & -y_0(z) \\ y_2'(z) & -y_0'(z) \end{pmatrix} 
\quad \textrm{for} \quad \arg z \in \left(0,\frac{\pi}{2}\right) \\
&\begin{pmatrix}y_2(z) & y_1(z) \\ y_2'(z) & y_1'(z) \end{pmatrix} 
\quad \textrm{for} \quad \arg z \in \left(\frac{\pi}{2},\pi \right) \\
&\begin{pmatrix}y_1(z) & -y_2(z) \\ y_1'(z) & -y_2'(z) \end{pmatrix} 
\quad \textrm{for} \quad \arg z \in \left(-\pi,-\frac{\pi}{2}\right) \\
&\begin{pmatrix}y_1(z) & y_0(z) \\ y_1'(z) & y_0'(z) \end{pmatrix} 
\quad \textrm{for} \quad \arg z \in \left(-\frac{\pi}{2},0 \right).
\end{aligned}
\right.
\end{equation}
Also introduce the function
\begin{equation}
\begin{aligned}
\psi(z):= - \left[\frac{3\pi}{2} \int_\phi^\be \left(\frac{\xi}{2\pi}-\rho(\theta)\right)\,d\theta\right]^{2/3}, \quad z=e^{i\phi}\in C_\be \cap D(e^{i\be},\ep). \\
\end{aligned}
\end{equation}
Since $\frac{\xi}{2\pi}-\rho(\th)$ vanishes at $\th=\be$ exactly like a square root, the function defined above is in fact analytic at $z=e^{i\be}$, and therefore extends to an analytic function on $D(e^{i\be},\ep)$. The function $\psi(z)$ is real valued on $C_\al$, $\psi_r(e^{i\be})=0$, and 
\begin{equation}
\frac{d}{d\phi} \psi\big(e^{i\phi}\big)\bigg|_{\phi=\be} >0.
\end{equation}
The solution to the local Riemann--Hilbert Problem in the disc $D(e^{i\be}, \ep)$ is given as
\begin{equation}
\mathbf U_M(z)=\mathbf E(z) \Phi(M^{2/3} \psi(z))e^{-M(g(z)-\frac{l}{2}-\frac{\log z}{2}-\frac{i\pi}{2})\sg_3} \times \left\{
\begin{aligned}
&z^{-(m/2)\sg_3}, \quad |z|<1 \\
&z^{(m/2)\sg_3}, \quad |z|>1, \\
\end{aligned}
\right.
\end{equation}
where
\begin{equation}
\mathbf E(z)=2\sqrt{\pi} \mathbf M(z) e^{\frac{im\al}{2}\sg_3} \begin{pmatrix} -i & -1 \\ -i & 1 \end{pmatrix}^{-1} M^{(1/6)\sg_3}\psi(z)^{(1/4)\sg_3}.
\end{equation}

The solution in $D(e^{-i\be},\ep)$ is similar and we do not present it here.

\subsection{The parametrix at the void-saturated region end points}\label{sec:hard_edge}
We now introduce a local transformation of the Riemann--Hilbert problem close to the endpoints $e^{\pm i\al}$ which allows for uniform estimates close to these points. The basic ideas behind this transformation can be found in \cite{Wang-Wong11}.
Introduce the function 
\begin{equation}
D(\z):=\left\{
\begin{aligned}
& \frac{\Ga\left(\frac{m\z}{2}+\frac{3}{2}\right)e^{m\z/2}}{\sqrt{2\pi}\left(\frac{m\z}{2}\right)^{m\z/2+1}}\,, \quad \textrm{for} \ \Re \z>0\,, \\ 
& \frac{\sqrt{2\pi}e^{m\z/2}}{\Ga\left(-\frac{m\z}{2}-\frac{1}{2}\right)\left(-\frac{m\z}{2}\right)^{m\z/2+1}}\,, \quad \textrm{for} \ \Re \z<0\,.
\end{aligned}\right.
\end{equation}
This function has the following properties:
\begin{itemize}
\item For $\z\in i\R$, the function $D$ has the multiplicative jump
\begin{equation} 
D_+(\z)=D_-(\z)\times\left\{
\begin{aligned}
&(1+e^{im\pi \z})\,, \quad \Im \z>0 \\
&(1+e^{-im\pi \z})\,, \quad \Im \z<0, \\
\end{aligned}\right.
\end{equation}
where the imaginary axis is oriented upward.
\item As $m\to \infty$, $D(\z) = 1+\bigO(m^{-1})$ for $\z$ bounded away from zero.
\end{itemize}
The first property follows from the reflection formula for the Gamma function, and the second follows from Stirling's formula. 

Now introduce the change of variable in a neighborhood of $e^{i\al}$,
\begin{equation}
\z_\al(z)=\frac{\log z-i\al}{i\pi}.
\end{equation}
Notice that the interval $I^+_\al$ is mapped to a piece of $i\R_+$, and $I^-_\al$ is mapped to a piece of $i\R_-$. In a small neighborhood to $e^{i\al}$, the arc $C_\al$ is mapped to a piece of the negative real axis, and the rest of $\T$ is mapped to the positive real axis. It follows that the composite function $D(\z_\al(z))$ has jumps on the intervals $I_\al^{\pm}$:
\begin{equation} 
D_+(\z_\al(z))=D_-(\z_\al(z))\times\left\{
\begin{aligned}
&(1+e^{m (\log(z)-i\al)})=(1-z^m)\,, \quad z\in I_\al^+ \\
&(1+e^{-m (\log(z)-i\al)})=(1-z^{-m})\,, \quad z\in I_\al^-, \\
\end{aligned}\right.
\end{equation}
where we have used the fact that $e^{i\al}=-1$.

Similarly, if we make the change of variable in a neighborhood of $e^{-i\al}$,
\begin{equation}
\z_{(-\al)}(z)=\frac{\log z+i\al}{i\pi},
\end{equation}
then the composite function $D(\z_{(-\al)}(z))$ has jumps on the intervals $I_{(-\al)}^{\pm}$:
\begin{equation} 
D_+(\z_{(-\al)}(z))=D_-(\z_{(-\al)}(z))\times\left\{
\begin{aligned}
&(1+e^{m (\log(z)+i\al)})=(1-z^m)\,, \quad z\in I_{(-\al)}^+ \\
&(1+e^{-m (\log(z)+i\al)})=(1-z^{-m})\,, \quad z\in I_{(-\al)}^-. \\
\end{aligned}\right.
\end{equation}
Define now the matrix function
\begin{equation} 
\mathbf D(z)=\left\{
\begin{aligned}
&D(\z_\al(z))^{\sg_3},\quad \textrm{for} \ z\in D(e^{i\al}, \ep)      \\
&D(\z_{(-\al)}(z))^{\sg_3},\quad \textrm{for} \ z\in D(e^{-i\al}, \ep)     \,.
\end{aligned}\right.
\end{equation}
It satisfies the jump condition on the intervals $I_\al^{\pm}$ and $I_{(-\al)}^{\pm}$,
\begin{equation} 
\mathbf D_+(z)=\mathbf D_-(z) \mathbf J_D(z)\,,
\end{equation}
where
\begin{equation}     
\mathbf J_D(z)=\left\{
\begin{aligned}
&(1-z^m)^{\sg_3}\,, \quad \textrm{for} \ z\in I_\al^+ \cup I_{(-\al)}^+ \\             
&(1-z^{-m})^{\sg_3}\,, \quad \textrm{for} \ z\in I_\al^- \cup I_{(-\al)}^-. \\    
\end{aligned}\right.
\end{equation}   

We now make the third transformation of the RHP, defining $\mathbf Y(z)$ as
\begin{equation}   \label{def:YM}
\mathbf Y(z):=\left\{
\begin{aligned}
&\mathbf S_M(z) \mathbf D(z) \,, \qquad \textrm{for} \ z\in D(e^{i\al}, \ep) \cup D(e^{-i\al}, \ep) \\
&\mathbf S_M(z), \qquad \textrm{otherwise}.
\end{aligned}\right.
\end{equation}
It has jumps on the contour $\Sg_R \cup \d  D(e^{i\al}, \ep) \cup \d D(e^{-i\al}, \ep)$,
\begin{equation} 
\mathbf Y_+(z)=\mathbf Y_-(z) \mathbf J_Y(z)\,,
\end{equation}
where
\begin{equation}     
\mathbf J_Y(z)=\left\{
\begin{aligned}
&\mathbf D_-(z)^{-1} \mathbf J_S(z) \mathbf D_+(z)\,, \quad \textrm{for} \ z\in \Sg_R \cap \{D(e^{i\al}, \ep) \cup D(e^{-i\al}, \ep)\} \\             
&\mathbf D(z)\,, \quad \textrm{for} \ z\in \d  D(e^{i\al}, \ep) \cup \d D(e^{-i\al}, \ep) \\    
&\mathbf J_S(z)\,, \quad \textrm{otherwise}.
\end{aligned}\right.
\end{equation}   
Consider the jump $ \mathbf J_S(z)$ for $z\in I_\al^\pm \cup I_{(-\al)}^\pm$.  By \eqref{st5a} it is
\begin{equation} 
\mathbf J_S(z)=\left\{
\begin{aligned}
&(1-z^m)^{-\sg_3} + \begin{pmatrix} 0 & 0 \\ -e^{-M(2g(z)-l-i\pi-\log(z))} & 0 \end{pmatrix}\,, \quad z\in I_\al^+ \cup I_{(-\al)}^+ \\
&(1-z^{-m})^{-\sg_3} + \begin{pmatrix} 0 & 0 \\ e^{-M(2g(z)-l-i\pi-\log(z))} & 0 \end{pmatrix}\,, \quad z\in I_\al^- \cup I_{(-\al)}^-.
\end{aligned}\right.
\end{equation}
Notice that the diagonal entries of this jump matrix are the reciprocals of the diagonal entries of the jump matrix for $\mathbf D(z)$. According to the Euler--Lagrange inequalities \eqref{eq15} together with \eqref{eq30}, the off-diagonal entry of this jump matrix is exponentially small in $M$. Thus we have
 \begin{equation} 
\mathbf J_S(z)=
\mathbf J_D(z)^{-1} + \begin{pmatrix} 0 & 0 \\ \bigO(e^{-cM}) & 0 \end{pmatrix}\,, \quad z\in I_\al^\pm \cup I_{(-\al)}^\pm\,, \\
\end{equation} 
for some constant $c>0$. The jump $\mathbf J_Y(z)$ is therefore
 \begin{equation} 
\mathbf J_Y(z)=\mathbf D_-(z)^{-1} \mathbf J_D(z)^{-1} \mathbf D_+(z)+\mathbf D_-(z)^{-1} \begin{pmatrix} 0 & 0 \\ \bigO(e^{-cM}) & 0 \end{pmatrix}\mathbf D_+(z) \quad \textrm{for} \ z\in I_\al^\pm \cup I_{(-\al)}^\pm\,.
\end{equation}
Since the matrices $\mathbf D_-(z)^{-1}, \mathbf J_D(z)^{-1},$ and  $\mathbf D_+(z)$ are all diagonal and therefore commute, and $\mathbf D(z)$ is uniformly bounded, we have
 \begin{equation} 
\mathbf J_Y(z)=\mathbf I+ \bigO(e^{-cM}), \quad \textrm{for} \ z\in I_\al^\pm \cup I_{(-\al)}^\pm\,.
\end{equation}
A similar estimate holds for the jump of $\mathbf Y(z)$ on $C_\al \cap \{D(e^{i\al}, \ep) \cup D(e^{-i\al}, \ep)\}$, so we have
 \begin{equation} 
\mathbf J_Y(z)=\mathbf I+ \bigO(e^{-cM}), \quad \textrm{for} \ z\in \Sg_R \cap \{D(e^{i\al}, \ep) \cup D(e^{-i\al}, \ep)\}\,.
\end{equation}
On the circles $ \d  D(e^{i\al}, \ep)$ and $\d D(e^{-i\al}, \ep)$ we have the estimate $\mathbf D(z)=\mathbf I+\bigO(m^{-1})=\mathbf I+\bigO(M^{-1})$, and thus $\mathbf J_Y(z)$ satisfies the estimates
\begin{equation}     
\mathbf J_Y(z)=\left\{
\begin{aligned}
&\mathbf  I+ \bigO(e^{-cM})\,, \quad \textrm{for} \ z\in \Sg_R \cap \{D(e^{i\al}, \ep) \cup D(e^{-i\al}, \ep)\} \\             
&\mathbf  I+\bigO(M^{-1})\,, \quad \textrm{for} \ z\in \d  D(e^{i\al}, \ep) \cup \d D(e^{-i\al}, \ep) \\    
&\mathbf J_S(z)\,, \quad \textrm{otherwise}.
\end{aligned}\right.
\end{equation}  

\subsection{The final transformation of the RHP and proofs of Propositions \ref{monic_band}, \ref{thmasym2}, \ref{hard_edge}, \ref{turning_points}, \ref{rho_M}, and \ref{asymhn}}
We consider the contour $\Sg_X$, which consists of the circles $\d D(e^{i\be}, \ep)$, $\d D(e^{-i\be}, \ep)$, $\d D(e^{i\al}, \ep)$, $\d D(e^{-i\al}, \ep)$, along with the part of $\Sg_R$ lies outside the disks $\d D(e^{\pm i\be}, \ep)$.

We make the transformation 
\begin{equation}\label{tt1}
\mathbf  X_M(z)=\left\{
\begin{aligned}
&\mathbf  Y_M(z) \mathbf  U_M(z)^{-1}\,, \quad z \in D(e^{\pm i\be},\ep) \\
&\mathbf  Y_M(z) \mathbf  M(z)^{-1}\,, \quad \textrm{otherwise}\,.
\end{aligned}\right.
\end{equation}
Then $\mathbf  X_M(z)$ solves the following RHP.
\begin{enumerate}
\item $\mathbf  X_M(z)$ is analytic on $\C \setminus \Sg_X$.
\item On the contour $\Sg_X$, $\mathbf  X(z)$ satisfies the jump properties
\begin{equation}\label{tt1a}
\mathbf  X_{M+}(z)=\mathbf  X_{M-}(z)\mathbf J_X(z),
\end{equation}
where
\begin{equation}\label{tt2}
\mathbf J_X(z)=\left\{
\begin{aligned}
&\mathbf  M(z) \mathbf  U_M(z)^{-1}\,, \quad z\in \partial D( e^{\pm i\be},\ep) \\
&\mathbf  M(z) \mathbf J_Y(z) \mathbf  M(z)^{-1}\,, \quad \textrm{otherwise}.
\end{aligned}\right.
\end{equation}
\item As $z\to\infty$, 
\begin{equation}\label{tt3}
\mathbf  X_M(z)= \mathbf I +\frac{\mathbf  X_1}{z} +\frac{\mathbf  X_2}{z^2}+\dots
\end{equation}
\end{enumerate}
Note that $\mathbf M(z)$ is analytic in a neighborhood of $e^{\pm i\al}$, so the same estimates which held for $\mathbf J_Y(z)$ also hold for $\mathbf J_X(z)$ close to $e^{\pm i\al}$. On the circles $\d D(e^{\pm i\be}, \ep)$, we have the estimate $\mathbf J_X(z)=\mathbf I+ \bigO(M^{-1})$, and thus
 $\mathbf J_X(z)$ is uniformly close to the identity.  More specifically, for some $c>0$, we have the uniform estimate
\begin{equation}\label{tt4}
\mathbf J_X(z)=\left\{
\begin{aligned}
&\mathbf I+\bigO(M^{-1})\,, \quad z\in \d D(e^{\pm i\al},\ep) \cup \d D(e^{\pm i\be},\ep)  \\
&\mathbf I+\bigO(e^{-cM}) \quad \textrm{otherwise}.
\end{aligned}\right.
\end{equation}
The solution to the RHP for $\mathbf  X_M(z)$ is well known and is given by a series of perturbation theory.  The solution is
\begin{equation}\label{ft5}
\begin{aligned}
\mathbf  X_M(z)&=\mathbf I+\sum_{k=1}^\infty \mathbf  X_{M,k}(z)\,, \quad \mathbf  X_{M,k}(z)=-\frac{1}{2\pi i}\int_{\Sg_X} \frac{\mathbf  X_{M,k-1}(u)(j_X(u)-\mathbf I)}{z-u}du \\
\mathbf  X_{M,0}(z)&=\mathbf I.
\end{aligned}
\end{equation}
Notice that, according to (\ref{tt4}), $\mathbf  X_{M,k}=\bigO(M^{-k})$, thus $\mathbf  X_M(z) = \mathbf I+\bigO(M^{-1})$ uniformly for $z\in \C$.

We can now prove Propositions \ref{monic_band}, \ref{thmasym2}, \ref{hard_edge}, \ref{turning_points}, and \ref{rho_M} by inverting the explicit transformations to the Interpolation Problem presented in Section \ref{sec:IP} to write a formula for $\mathbf P_M(z)$ in terms of $\mathbf  X_M(z)$. Since we made different transformations in various regions of the complex plane, we arrive at different formulas for $\mathbf P_M(z)$. For $z$ bounded away from the arc $C_\al$, we use \eqref{red5}, \eqref{ft1}, \eqref{st1}, \eqref{def:YM}, and \eqref{tt1} to find
\begin{equation}
\mathbf P_M(z) = \sg_3 e^{M\frac{l+i\pi}{2}\sg_3} \mathbf X_M(z) \mathbf M(z) e^{M(g(z)-\frac{l+i\pi}{2})\sg_3}\sg_3.
\end{equation}
Using  $\mathbf  X_M(z) = \mathbf I+\bigO(M^{-1})$, we can expand this expression at take the $(11)$-entry to prove Proposition \ref{rho_M}. Plugging in $z=0$ and using \eqref{IP4a} then proves Proposition \ref{asymhn}.

Similarly, for $z$ in a neighborhood of a compact subset of the open arc $C_\be$ we have
\begin{equation}
\mathbf P_M(z) = \sg_3 e^{M\frac{l+i\pi}{2}\sg_3} \mathbf X_M(z) \mathbf M(z) \mathbf J_{\pm}(z)^{\pm 1}e^{M(g(z)-\frac{l+i\pi}{2})\sg_3}\sg_3\mathbf D_{\pm}^u(z)^{-1}, \qquad \textrm{for} \ \pm(|z|-1)<0.
\end{equation}
Proposition \ref{monic_band} is proved by multiplying out this explicit formula, using $\mathbf  X_M(z) = \mathbf I+\bigO(M^{-1})$, taking the limit as $|z|\to 1$ from either side, and looking at the $(11)$-entry. 

For $z$ in a neighborhood of a compact subset of $\{e^{\pm i\phi} : \be\le\phi\le\al\}$ we have
\begin{equation}
\mathbf P_M(z) = \mp \sg_3 e^{M\frac{l+i\pi}{2}\sg_3} \mathbf X_M(z) \mathbf M(z) \mathbf z^{\pm (m/2)\sg_3}e^{M(g(z)-\frac{l+i\pi}{2})\sg_3}\sg_3\mathbf D_{\pm}^l(z)^{-1}, \qquad \textrm{for} \ \pm(|z|-1)<0.
\end{equation}
Proposition \ref{thmasym2} is proved by multiplying out this explicit formula, using $\mathbf  X_M(z) = \mathbf I+\bigO(M^{-1})$, taking the limit as $|z|\to 1$ from either side, and looking at the $(11)$-entry. 

Finally, Propositions   \ref{hard_edge} and \ref{turning_points}  are proved by inverting the explicit transformations in a neighborhood of the points $e^{\pm i\al}$ and $e^{\pm i\be}$, respectively. These transformations involve the local solutions presented in Sections  \ref{sec:hard_edge} and \ref{sec:airy}. The transformations are different in different sectors around the points $e^{\pm i\al}$ and $e^{\pm i\be}$, but one can check that the different transformations give a uniform asymptotic formula in a full neighborhood of $e^{\pm i\al}$ and $e^{\pm i\be}$. We  omit the lengthy but straightforward and standard calculation.

\bibliographystyle{plain}
\bibliography{bibtex.bib}

\end{document}